\documentclass[11pt]{article}
\usepackage{amsmath,latexsym,amssymb,amsfonts,amsbsy, amsthm}
\usepackage{graphicx}
\usepackage{ulem}
\def\inte#1{
\displaystyle\mathop{#1\kern0pt}^\circ }

\def\virgp{\raise 2pt\hbox{,}}
\def\cdotpv{\raise 2pt\hbox{;}}

\def\C{\mathop{\bf C\kern 0pt}\nolimits}
\def\DD{\mathop{\bf D\kern 0pt}\nolimits}
\def\K{\mathop{\bf K\kern 0pt}\nolimits}
\def\N{\mathop{\bf N\kern 0pt}\nolimits}
\def\Q{\mathop{\bf Q\kern 0pt}\nolimits}
\def\R{\mathop{\bf R\kern 0pt}\nolimits}
\def\SS{\mathop{\bf S\kern 0pt}\nolimits}
\def\ZZ{\mathop{\bf Z\kern 0pt}\nolimits}
\def\TT{\mathop{\bf T\kern 0pt}\nolimits}

\newcommand{\LC}{\left(}
\newcommand{\RC}{\right)}

\newcommand{\LCB}{\left\{}
\newcommand{\RCB}{\right\}}

\newcommand{\beq}{\begin{equation}}
\newcommand{\eeq}{\end{equation}}
\newcommand{\ben}{\begin{eqnarray}}
\newcommand{\een}{\end{eqnarray}}
\newcommand{\beno}{\begin{eqnarray*}}
\newcommand{\eeno}{\end{eqnarray*}}
\newtheorem{remark}{Remark}[section]
\newtheorem{lemma}{Lemma}[section]
\newtheorem{definition}{Definition}[section]

\newtheorem{theorem}{Theorem}[section]

\renewcommand{\theequation}{\thesection.\arabic{equation}}

\begin{document}

%%%%%%%%%%%%%%%%%%%%%%%%%%%%%%%%%%%%%%%%%%%%%%%%%%%%%%%%%%%%%%%%%%
%% %%%%                                                                            Introduction                                                                             %%%%%%

\title{On a shallow-water approximation to the Green-Naghdi equations with the Coriolis effect}
\author{
Robin Ming Chen \footnote{Department of Mathematics, University of Pittsburgh, Pittsburgh, PA 15260; {mingchen@pitt.edu}.} \and
Guilong Gui \footnote{Center for Nonlinear Studies, School of Mathematics, Northwest University, Xi'an 710069, China; {glgui@amss.ac.cn}.} \and
Yue Liu \footnote{Department of Mathematics, University of Texas at Arlington, Arlington, TX 76019; {yliu@uta.edu}.}
}

\date{}
\maketitle

%%%%P13 address changed (we have some new administrative rules)

%\date{}%
%\dedicatory{}%
%\commby{}%
% --------------------------------------------------------------
\begin{abstract}

We consider  an  asymptotic 1D (in space) rotation-Camassa-Holm (R-CH)  model,  which could be used to describe the propagation of  long-crested shallow-water waves  in the equatorial ocean regions with allowance for the weak  Coriolis effect due to the Earth's rotation. This  model equation has similar wave-breaking phenomena as the Camassa-Holm equation.   It is  analogous to  the rotation-Green-Naghdi (R-GN) equations  with the weak Earth's rotation effect,  modeling the propagation of wave  allowing large amplitude in shallow water. We provide here a rigorous justification showing that  solutions of the R-GN equations tend to associated solution of the R-CH  model equation in the Camassa-Holm regime with the small amplitude  and the larger wavelength. Furthermore, we demonstrate that the R-GN model equations are locally well-posed in a Sobolev space  by the refined energy estimates.

\end{abstract}

% -----------------------------------------


\vskip 0.2cm


%%%%%%%%%%%%%%%%%%%%%%%%%%%%%%%%%%%%%%%%%%%%%%%%%%%%%%%%%%%%%%
%%%%%%%%%%%%%%%%%%%%%%%%%%%%%%%%%%%%%%%%%%%%%%%%%%%%%%%%%%%%%
\renewcommand{\theequation}{\thesection.\arabic{equation}}
\setcounter{equation}{0}

%%%%%%%%%%%%%%%%%%%%%%%%%%%%%%%%%%%%%%%%%%%%%%%%%%%%%%%%%%%%%%
%%%%%%%%%%%%%%%%%%%%%%%%%%%%%%%%%%%%%%%%%%%%%%%%%%%%%%%%%%%%%%%%%%
\section{Introduction}\label{sec_intro}

The theory of water waves embodies the Euler equations of fluid mechanics along with the crucial  behavior of boundaries. Due to the complexity and the difficulties arising in the theoretical and numerical study, simpler model equations have been proposed as approximations to the Euler equations in some specific physical regimes.

Among various asymptotic systems, one of the most prominent examples, which is widely used to model and numerically simulate the propagation of surface waves, in particular in coastal oceanography, is the Green-Naghdi equations (GN) \cite{GN76} (also known as the Serre \cite{Serre} or Su-Gardner equations \cite{SG}). The GN equations model the fully nonlinear shallow-water waves whose amplitude is not necessarily small and represent a higher-order correction to the classical shallow-water equations. The physical validity of the model depends on the characteristics of the flow under consideration. More precisely, it depends on particular assumptions made on the dimensionless parameters $\varepsilon$ and $\mu$ defined as
\begin{equation*}
\text{nonlinearity }\ \ \varepsilon := {a\over h_0}, \qquad \text{shallowness }\ \ \mu := {h_0^2\over \lambda^2},
\end{equation*}
where $a$ is the typical amplitude of the waves, $h_0$ is the mean depth, and $\lambda$ is the typical wavelength. The {\it shallow-water} (or long-wave) regime is characterized by the presumption of small depth or long wavelength ($\mu \ll 1$) only. In such a scaling regime and without any smallness assumption on $\varepsilon$, one can derive the following 1D GN equations
\begin{equation}\label{GN}
\begin{cases}
\eta_t + \left ((1+\varepsilon \eta) u\right )_x = 0,\\
\displaystyle u_t + \eta_x + \varepsilon u u_x = \frac{\mu}{3(1+\varepsilon \eta)}\left[ (1+\varepsilon \eta)^3(u_{xt}+\varepsilon u u_{xx} -\varepsilon u_{x}^2) \right]_x,
\end{cases}
\end{equation}
with an $O(\mu^2)$ correction, where $\eta(t,x)$ and $u(t,x)$ are the parameterization of the surface and the vertically
averaged horizontal component of the velocity at time $t$, respectively. A rigorous justification of the GN model can be found in \cite{Li} for the 1D water waves with a flat bottom; the general case was handled in \cite{alla,DuIs} based on a well-posedness theory.

The GN equation can serve as a seed model for many other asymptotic shallow-water models as far as the nonlinearity parameter $\varepsilon$ is concerned: the {\it weakly nonlinear} regime (also referred to as the Boussinesq scaling)
\begin{equation}\label{Bous scaling}
\mu \ll 1, \qquad \varepsilon = O(\mu)
\end{equation}
yields the usual Boussinesq models \cite{BCS, Bous,C85,KN},  the Korteweg-de Vries (KdV) \cite{KdV} and Benjamin-Bona-Mahoney (BBM) \cite{BBM} equations in the unidirectional case. Despite the advantage that these models provide good asymptotic approximations to the full water wave problem in the weakly nonlinear regime, they fail to capture observed some interesting wave phenomena in nature such as wave breaking and waves of greatest height \cite{AFT,Wh}. This motivates one to pursue alternative shallow-water models incorporating stronger nonlinearity. A successful attempt was made in the {\it moderately nonlinear} regime (known as the Camassa-Holm scaling)
\begin{equation}\label{CH scaling}
\mu \ll 1, \qquad \varepsilon = O(\sqrt{\mu}),
\end{equation}
giving rise to the Camasa-Holm (CH) and Degasperis-Procesi (DP) \cite{CH,CL09,DP} equations.

%
%one of the primary motivations in pursuing alternative shallow-water models incorporating stronger nonlinearity is due to the failure of weakly nonlinear model equations to capture observed interesting wave phenomena in nature such as wave breaking and waves of greatest height \cite{AFT,Wh}.
%
%The unidirectional limit in the {\it moderately nonlinear} regime (known as the Camassa-Holm scaling)
%\begin{equation}\label{CH scaling}
%\mu \ll 1, \qquad \varepsilon = O(\sqrt{\mu})
%\end{equation}
%gives rise to the Camasa-Holm (CH) and Degasperis-Procesi (DP) \cite{CH,CL09,DP} equations.

One of the goals in the present study is to re-derive model equations \cite{GLL16, GLS16} related to the classical Camassa-Holm equation that can also account for configurations where the effect of solid-body rotation of the Earth, namely the Coriolis effect, is present. Field observations reveal that the interaction between the gravity and the Earth's rotation may play a significant role in study of large-scale oceanic and atmospheric flows, leading to complex phenomena over wide ranges of length and time scales \cite{CB,GSR07,Ped,Vallis}. Simplified and approximate models therefore play a crucial role in potentially gaining insight into processes that occur in the full governing equations. In particular, we will neglect centripetal forces since they are relatively much smaller than the Coriolis terms. We also neglect the variations of the Coriolis parameter and employ the $f$-plane approximation which is applicable for oceanic flows restricted to a meridional range of small latitudinal deviation (about $2^\circ$) near the Equator \cite{Con12,ConJo,CB,Ped}.

%Equatorial ocean dynamics comprises a vast variety of phenomena, many of which are related to the presence of the strong eastward flowing Equatorial Undercurrent (EUC) \cite{FB,Iz}.

Several attempts have been made recently in deriving shallow water asymptotic models for free surface water waves under the influence of the gravity and Coriolis forcing using $f$-plane approximation; see, for e.g. \cite{Fan, GLS16}. These models can describe unidirectional waves in the moderately nonlinear scaling regime, resulting in the CH-type equations, referred to as the rotation-Camass-Holm (R-CH) equations. The approach taken in these works follows the classical idea of asymptotic perturbation analysis \cite{Iv,Jo1}. Our goal here is to put these formal asymptotic procedure on a firm and mathematically rigorous basis. The idea is in the general spirit of \cite{CL09}. More precisely, we will prove the relevance of the R-CH equation as a valid model for the propagation of shallow-water waves with effect of the Coriolis forcing. To do so, we will use the following rotation-Green-Naghdi (R-GN) equations with the Coriolis effect (see \cite{GLL16} for derivation of the model)
\begin{equation}\label{R-GN-2}
\begin{cases}
\eta_t + \left ((1+\varepsilon \eta) u\right )_x = 0,\\
u_t + \eta_x + \varepsilon u u_x + 2 \Omega \eta_t = \frac{\mu}{3(1+\varepsilon \eta)}\left ((1+\varepsilon \eta)^3(u_{xt}+\varepsilon u u_{xx} -\varepsilon u_{x}^2) \right )_x
\end{cases}
\end{equation}
as the reference system, where $ \Omega $ is the constant rotational frequency due to the Coriolis effect, and which is built on a one-parameter family of approximate equations {\it consistent} with the the R-GN equations in the sense of Definition \ref{def-consistent-CH-GN}. Further replacing the vertically averaged velocity by the horizontal velocity evaluated at a certain depth introduces an additional parameter which allows one to arrive at the following R-CH equation (cf. Theorem \ref{prop-two-parameters} and Remark \ref{rmk-CH-eqns-1})
\begin{equation}\label{R-CH-1}
\begin{split}
 u_t-\beta\mu  u_{xxt} + c u_x + 3\alpha\varepsilon uu_x - \beta_0\mu u_{xxx} &+ \omega_1 \varepsilon^2u^2u_x + \omega_2 \varepsilon^3u^3u_x   \\
&= \alpha\beta\varepsilon\mu( 2u_{x}u_{xx}+uu_{xxx}),
\end{split}
\end{equation}
where the constants are given as
\begin{equation*}
\begin{split}
& c = \sqrt{1 + \Omega^2} - \Omega, \quad  \alpha {=} \frac{c^2}{1+c^2}, \quad \beta_0  {=}\frac{c(c^4+6c^2-1)}{6(c^2+1)^2},  \quad \beta {=}\frac{3c^4+8c^2-1}{6(c^2+1)^2}, \\
& \omega_1 {=}\frac{-3c(c^2-1)(c^2-2)}{2(1+c^2)^3}, \quad \omega_2 {=}\frac{(c^2-2)(c^2-1)^2(8c^2-1)}{2(1+c^2)^5}.
\end{split}
\end{equation*}
Note that in the vanishing Coriolis force limit $\Omega \to 0$ we have
\begin{equation*}
c \to 1, \quad \alpha \to {1\over2}, \quad \beta_0 \to {1\over4}, \quad \beta \to {5\over12},  \quad \omega_1, \; \omega_2 \to 0.
\end{equation*}

Unraveling the change of unknowns \eqref{eta u} to express the velocity $u$ in terms of the surface variable $\eta$ with a higher order correction term, one may follow the same procedure as before to derive an equation for the surface evolution, cf. \eqref{free-GBBM-1}.

Our {\it convergence} result, Theorem \ref{thm-justification-1}, concerns comparison of solutions between  the R-GN and R-CH equations. It states that corresponding to a certain family $\{ (u^{\varepsilon, \mu}, \eta^{\varepsilon, \mu}) \}$ of solutions to the generalized BBM equations (including the R-CH solutions) that are consistent with the R-GN equations, there exists a family of solutions $(\underline{u}^{\varepsilon, \mu}, \underline{\eta}^{\varepsilon, \mu})$ to the R-GN equations satisfying
\[
\|\underline{u}^{\varepsilon, \mu}-u^{\varepsilon, \mu}\|_{L^{\infty}([0, t] \times \mathbb{R})}+\|\underline{\eta}^{\varepsilon, \mu}-\eta^{\varepsilon, \mu}\|_{L^{\infty}([0, t] \times \mathbb{R})} \leq O(\mu^2\,t)
\]
over a large time scale $t \in [0, T/\varepsilon]$. Therefore a preliminary well-posedness theory for the  R-GN and R-CH equations on $[0, T/\varepsilon]$ is need. The large time well-posedness for the R-CH equation has been established in \cite{GLS16}. Here we prove the case for the R-GN equations using an energy-type argument, cf. Theorem \ref{thm-local}. The error estimate is established owing to the symmetrizability of the R-GN equations, from which a stability result can be deduced, cf. Lemma \ref{lem-uniqueness-1}. Note that in absence of the Coriolis force, under the CH scaling \eqref{CH scaling}, the solutions to the GN equations approximate the solutions of the full water wave equations with an $O(\mu^2 t)$ error over a time scale $O(1/\varepsilon)$ \cite {alla}. Therefore it is reasonable to expect that the R-GN equations would give a correct approximation to the $f$-plane rotating full water wave model with the same precision, and hence it would follow that the R-CH model approximates the full $f$-plane water wave equations with the same accuracy. But the justification of the R-GN equations is beyond the scope of this article, and will be discussed in a forthcoming paper.

The rest of the paper is organized as follows. In Section \ref{Sec_rch}, we will derive a family of equations for the horizontal velocity $u$ which are consistent with the R-GN equations. Among such a family the R-CH equation can be recovered by introducing the depth parameter. Following the same procedure, in Section \ref{Sec_surface} we derive the corresponding family of equations governing the evolution of the surface. In Section \ref{Sec_local},  we prove that the R-GN equations are locally well-posed in the energy space over an $O(1/\varepsilon)$ time period. Finally in Section \ref{Sec_just}, we demonstrate the result on the asymptotic  convergence of the R-CH equation to the R-GN equations.

%%%%%%%%%%%%%%%%%%%%%%%%%%%%%%%%%%%%%%%%%%%%%%%%%%%%%%%%%%%%%%
%%%%%%%%%%%%%%%%%%%%%%%%%%%%%%%%%%%%%%%%%%%%%%%%%%%%%%%%%%%%%
\renewcommand{\theequation}{\thesection.\arabic{equation}}
\setcounter{equation}{0}

%%%%%%%%%%%%%%%%%%%%%%%%%%%%%%%%%%%%%%%%%%%%%%%%%%%%%%%%%%%%%%
%%%%%%%%%%%%%%%%%%%%%%%%%%%%%%%%%%%%%%%%%%%%%%%%%%%%%%%%%%%%%%%%%%
\section{Derivation of the rotation-Camassa-Holm equation}\label{Sec_rch}
In this section, we derive asymptotical equations to the rational-Green-Naghdi equations in the
Camassa-Holm regime, that is, parameters $\varepsilon$ and $\mu$ belong to the class
\begin{equation}\label{CH-regime}
\mathcal{P}_{\mu_0, M} = \LCB\LC\varepsilon, \mu\RC\vert~0<\mu\leq\mu_0, 0 <\varepsilon\leq M\sqrt{\mu}\RCB,
\end{equation}
for given constants $\mu_0$,  $M>0$.

\begin{definition}
\label{def-consistent-CH-GN}
Let $\mu_0$,  $M>0$, $T>0$ and $\mathcal{P}_{\mu_0, M} $ be as defined in \eqref{CH-regime}. A family $\{\eta^{\varepsilon, \mu}, u^{\varepsilon, \mu}\}_{(\varepsilon, \mu)\in \mathcal{P}_{\mu_0, M}}$ is consistent (of order $s \geq 0$ and on $[0, \frac{T}{\varepsilon}]$) with the R-GN equations \eqref{R-GN-2} if for all $(\varepsilon, \mu)\in \mathcal{P}_{\mu_0, M}$,
\begin{equation}\label{R-GN-consistent-1}
\begin{cases}
\eta_t + \left ((1+\varepsilon \eta) u\right )_x = \mu^2 r_1^{\varepsilon, \mu},\\
u_t + \eta_x + \varepsilon u u_x + 2 \Omega \eta_t = \frac{\mu}{3(1+\varepsilon \eta)}\left ((1+\varepsilon \eta)^3(u_{xt}+\varepsilon u u_{xx} -\varepsilon u_{x}^2) \right )_x+\mu^2 r_2^{\varepsilon, \mu}
\end{cases}
\end{equation}
with $(r_1^{\varepsilon, \mu}, r_2^{\varepsilon, \mu})_{(\varepsilon, \mu)\in \mathcal{P}_{\mu_0, M}}$ bounded in $L^{\infty}([0, \frac{T}{\varepsilon}], H^s(\mathbb{R}))$.

\end{definition}
For the sake of simplicity,  given parameters $\alpha$, $\beta$, $\gamma$, and $\delta$, we denote some coefficients as follows:
\begin{equation}\label{def-coeff-1}
  \begin{split}
  & c= \sqrt{1+\Omega^2}-\Omega, \quad \omega_1 = \frac{-3c(c^2-1)(c^2-2)}{2(1+c^2)^3}, \quad \omega_2 =\frac{(c^2-2)(c^2-1)^2(8c^2-1)}{2(1+c^2)^5},\\
  & A_1= \left(\frac{1}{3}+\gamma+\frac{3\alpha\,c}{c^2+1} \right)+2\Omega \left(\frac{1}{3}+\beta-\frac{\alpha}{c} \right), \quad A_2= \frac{1+\delta-\gamma}{2}+\frac{3\alpha\,c}{c^2+1},\\
  &A_3=\frac{\omega_1}{3}-2\Omega h,\quad A_4=c\left(A_1+ \left(\frac{1}{3}+\gamma+\frac{\alpha}{c} \right) \frac{3\alpha\,c}{c^2+1}-2h \left(\beta-\frac{\alpha}{c} \right) \right),\\
  &A_5= c \left(A_2+ \left(\frac{1}{3}+\gamma+\frac{\alpha}{c} \right) \frac{3\alpha\,c}{c^2+1}+h \left(\beta-\frac{\alpha}{c} \right) \right).
\end{split}
\end{equation}

The following theorem  shows that there is a one parameter family of equations
consistent with the R-GN equations.

\begin{theorem}\label{prop-one-parameter}
Let $p \in \mathbb{R}$, and assume that
\begin{equation}\label{derivation-CH-16}
\begin{split}
&\alpha = c\,p, \quad \beta=-\frac{c^2}{3(c^2+1)}+p, \quad \gamma=-\frac{c^2(5c^2-1)}{3(c^2+1)^3}-\frac{3c^2}{c^2+1}p,\\
& \delta=-\frac{c^2(3c^4+16c^2+4)}{3(c^2+1)^3}-\frac{9c^2}{c^2+1}p.
\end{split}
\end{equation}
Then there exists $D>0$ such that:
for all $s \geq 0$ and $T>0$, and for all bounded family $(u^{\varepsilon, \mu})_{(\varepsilon, \mu)\in \mathcal{P}_{\mu_0, M}} \in L^{\infty}([0, \frac{T}{\varepsilon}], H^{s+D}(\mathbb{R}))$ solving
\begin{equation}\label{GBBM-1}
\begin{split}
&u_t+c u_x+ 3\frac{c^2}{c^2+1}\varepsilon uu_x+\omega_1\varepsilon^2 u^2u_x+\omega_2\varepsilon^3 u^3u_x+\mu(\alpha u_{xxx}+\beta u_{xxt})\\
&=\varepsilon\mu(\gamma uu_{xxx}+\delta u_xu_{xx}),
\end{split}
\end{equation}
the family $\{\eta^{\varepsilon, \mu}, u^{\varepsilon, \mu}\}_{(\varepsilon, \mu)\in \mathcal{P}_{\mu_0, M}}$, with (omitting the indexes $\varepsilon, \mu$)
\begin{equation}\label{relation-eta-u-1}
\begin{split}
\eta&=F(u) :=\frac{1}{c}u- \varepsilon h\, u^2+A_3 \varepsilon^2 \,u^3+ \left(2\Omega A_3+\frac{\omega_2}{4} \right)\varepsilon^3 \,u^4 \\
& \quad \qquad\qquad +\mu \left(\frac{1}{3}+\beta-\frac{\alpha}{c} \right)\, u_{xt}-\varepsilon\mu\bigg(A_1 uu_{xx}+ A_2 u_x^2\bigg)
\end{split}
\end{equation}
is consistent (of order $s$ and on $[0, \frac{T}{\varepsilon}]$) with the R-GN equations \eqref{R-GN-2}.
\end{theorem}

\begin{proof}
For the sake of simplicity, we use the notation $O(\mu)$, $O(\mu^2)$ etc., without
explicit mention to the functional normed space to which we refer. A precise statement
has been given in Definition \ref{def-consistent-CH-GN}. All the equalities would be understood in the appropriate mixed time-spatial Sobolev spaces throughout the proof.

Taking
\begin{equation}\label{eta u}
\eta = \frac{1}{c}u+\varepsilon v,
\end{equation}
where $v$ will be determined later on, we get from the first equation in \eqref{R-GN-2} that
\begin{equation}\label{derivation-CH-1}
\begin{split}
u_t+cu_x+c(\varepsilon\, v)_t+\varepsilon (u^2)_x+ c\varepsilon(\varepsilon\, uv)_x=0.
\end{split}
\end{equation}
Plugging $\eta= \frac{1}{c}u+\varepsilon v$ into the second equation of\eqref{R-GN-2}, we have
\begin{equation}\label{derivation-CH-2}
\begin{split}
&u_t+\frac{1}{c}u_x+\varepsilon\, v_x-2\Omega u_x-2\Omega \varepsilon \frac{1}{c} (u^2)_x-2\Omega \varepsilon (\varepsilon\, uv)_x+\frac{1}{2}\varepsilon (u^2)_x\\
&=\frac{\mu}{3}u_{xxt}-\frac{\mu \varepsilon}{3} \left(uu_{xx}+\frac{3}{2}u_x^2 \right)_x,
\end{split}
\end{equation}
which implies
\begin{equation}\label{derivation-CH-3}
\begin{split}
&u_t+c\, u_x+\varepsilon\, v_x+\varepsilon \left(\frac{1}{2}-\frac{2\Omega}{c} \right)(u^2)_x-2\Omega \varepsilon (\varepsilon\, uv)_x=\frac{\mu}{3}u_{xxt}-\frac{\mu \varepsilon}{3}\partial_x (uu_{xx}+\frac{3}{2}u_x^2).
\end{split}
\end{equation}
On the other hand, we expect $u$ satisfying the following generalized BBM equation \eqref{GBBM-1}.
To this end,  up to the $O(\varepsilon^2)$ terms (where $\varepsilon=O(\sqrt{\mu})$), we first get
\begin{equation}\label{derivation-CH-5}
\begin{split}
&u_t+c u_x+ 3\frac{c^2}{c^2+1}\varepsilon uu_x=O(\varepsilon^2, \mu),
\end{split}
\end{equation}
which gives rise to
\begin{equation}\label{derivation-CH-6}
\begin{split}
&u_{xxx}=-\frac{1}{c} \left(u_{xxt}+ 3\frac{c^2}{c^2+1}\varepsilon (uu_x)_{xx} \right)=O(\varepsilon^2, \mu).
\end{split}
\end{equation}
Hence, we can replace the term $u_{xxx}$ in \eqref{GBBM-1} by this expression to get
\begin{equation}\label{derivation-CH-7}
\begin{split}
&u_t+c u_x+ 3\frac{c^2}{c^2+1}\varepsilon uu_x+\mu \left(\beta-\frac{\alpha}{c} \right)u_{xxt}+\omega_1\varepsilon^2 u^2u_x+\omega_2\varepsilon^3 u^3u_x\\
&=\varepsilon\mu \left( \left(\gamma+\frac{3\alpha\,c}{c^2+1} \right) uu_{xx}+ \left(\frac{\delta-\gamma}{2}+\frac{3\alpha\,c}{c^2+1} \right)u_x^2\right)_x+O(\mu\varepsilon^2, \mu^2).
\end{split}
\end{equation}
Combining \eqref{GBBM-1} with \eqref{derivation-CH-7}, we have
\begin{equation}\label{derivation-CH-8}
\begin{split}
&\varepsilon v_x+\varepsilon h\, (u^2)_x-2\Omega \varepsilon(u\,\varepsilon\,v)_x-\frac{\omega_1}{3}\varepsilon^2 (u^3)_x-\frac{1}{4}\omega_2\varepsilon^3 (u^4)_x\\
&= \left(\frac{1}{3}+\beta-\frac{\alpha}{c} \right)\mu\, u_{xxt}-\varepsilon\mu \left( \left(\frac{1}{3}+\gamma+\frac{3\alpha\,c}{c^2+1} \right) uu_{xx}+ \left(\frac{1+\delta-\gamma}{2}+\frac{3\alpha\,c}{c^2+1} \right)u_x^2\right)_x\\
&\quad +O(\mu\varepsilon^2, \mu^2),
\end{split}
\end{equation}
where $h = \frac{1}{2}-\frac{2\Omega}{c}-\frac{3c^2}{2(c^2+1)}=\frac{c^2-2}{2c^2(c^2+1)}$, which implies by integrating \eqref{derivation-CH-8} with respect to $x$ that
\begin{equation}\label{derivation-CH-9}
\begin{split}
&(1-2\Omega \varepsilon\, u)\varepsilon v =- \varepsilon h\, u^2+\frac{\omega_1}{3}\varepsilon^2 \,u^3+\frac{1}{4}\omega_2\varepsilon^3 \,u^4+ \left(\frac{1}{3}+\beta-\frac{\alpha}{c} \right)\mu\, u_{xt}\\
&-\varepsilon\mu \left( \left(\frac{1}{3}+\gamma+\frac{3\alpha\,c}{c^2+1} \right) uu_{xx}+ \left(\frac{1+\delta-\gamma}{2}+\frac{3\alpha\,c}{c^2+1} \right)u_x^2 \right) +O(\mu\varepsilon^2, \mu^2),
\end{split}
\end{equation}
where we may take the integration constant to be zero since all the terms in \eqref{derivation-CH-9} go to zero as $|x|\rightarrow \infty$.

It then follows that
\begin{equation}\label{derivation-CH-10}
\begin{split}
&\varepsilon v =(1+2\Omega \varepsilon\, u+(2\Omega \varepsilon\, u)^2+O(\varepsilon^3))\bigg(- \varepsilon h\, u^2+\frac{\omega_1}{3}\varepsilon^2 \,u^3+\frac{\omega_2}{4}\varepsilon^3 \,u^4\\
&+ \left(\frac{1}{3}+\beta-\frac{\alpha}{c} \right)\mu\, u_{xt}-\varepsilon\mu\bigg( \left(\frac{1}{3}+\gamma+\frac{3\alpha\,c}{c^2+1} \right) uu_{xx}+ \left(\frac{1+\delta-\gamma}{2}+\frac{3\alpha\,c}{c^2+1} \right)u_x^2\bigg) \\
&\qquad\qquad\qquad\qquad\qquad\qquad\qquad\qquad\qquad\qquad\qquad+O(\mu\varepsilon^2, \mu^2),
\end{split}
\end{equation}
which implies that
\begin{equation}\label{derivation-CH-11}
\begin{split}
&\varepsilon v =- \varepsilon h\, u^2+A_3 \varepsilon^2 \,u^3+ \left(2\Omega A_3+\frac{\omega_2}{4} \right)\varepsilon^3 \,u^4 \\
& \qquad\qquad+\mu \left(\frac{1}{3}+\beta-\frac{\alpha}{c} \right)\, u_{xt}-\varepsilon\mu\bigg(A_1 uu_{xx}+ A_2 u_x^2\bigg) +O(\mu\varepsilon^2, \mu^2),
\end{split}
\end{equation}
where $A_1$, $A_2$, and $A_3$ are defined in \eqref{def-coeff-1}.
Thanks to \eqref{GBBM-1} and \eqref{derivation-CH-6}, we derive
\begin{equation}\label{derivation-CH-12}
\begin{split}
\varepsilon v_t = & - \varepsilon h\, (u^2)_t+A_3 \varepsilon^2 \,(u^3)_t+ \left(2\Omega A_3+\frac{\omega_2}{4} \right)\varepsilon^3 \,(u^4)_t +\mu \left(\frac{1}{3}+\beta-\frac{\alpha}{c} \right)\, u_{xtt}\\
& -\varepsilon\mu\bigg(A_1 uu_{xx}+ A_2 u_x^2\bigg)_t +O(\mu\varepsilon^2, \mu^2) \\
= & \ 2 \varepsilon h\,c\, uu_x+ \left(2h\frac{3c^2}{c^2+1}-3A_3c \right) \varepsilon^2 \,u^2u_x\\
&+ \left(2h\omega_1-3A_3\frac{3c^2}{c^2+1}-c(8\Omega A_3+\omega_2) \right)\varepsilon^3 \,u^3u_x -\mu\,c \left(\frac{1}{3}+\beta-\frac{\alpha}{c}\right)\, u_{xxt}\\
& +\varepsilon\mu\bigg(A_4 uu_{xx}+ A_5 u_x^2\bigg)_x +O(\mu\varepsilon^2, \mu^2),
\end{split}
\end{equation}
where $A_4$ and $A_5$ are given in \eqref{def-coeff-1}.

Substituting \eqref{derivation-CH-11} and \eqref{derivation-CH-12} into \eqref{derivation-CH-1} yields
\begin{align}
%\begin{split}
& u_t+c u_x+ 3\frac{c^2}{c^2+1}\varepsilon uu_x+c \left(2h\frac{c^2}{c^2+1}-3A_3c -3h \right)\varepsilon^2 u^2u_x \nonumber\\
& +c \left(2h\omega_1-3A_3\frac{c^2}{c^2+1}-c(8\Omega A_3+\omega_2)+4A_3 \right)\varepsilon^3 u^3u_x -c^2 \left(\frac{1}{3}+\beta-\frac{\alpha}{c} \right)\mu\, u_{xxt} \nonumber \\
& =  \ \varepsilon\mu \left(\left(c^2 \left(\frac{1}{3}+\beta-\frac{\alpha}{c}\right)-cA_4 \right) uu_{xx}-cA_5 u_x^2 \right)_x. \label{derivation-CH-13}
%\end{split}
\end{align}
Comparing \eqref{derivation-CH-7} with \eqref{derivation-CH-13}, we take $\omega_1$, $\omega_2$, $\alpha$, $\beta$, $\gamma$, and $\delta$ in \eqref{GBBM-1} to satisfy the relation
\begin{equation*}\label{derivation-CH-14a}
\begin{split}
&\omega_1 = c \left(2h\frac{c^2}{c^2+1}-3A_3c -3h \right),\, \ \   \omega_2 = c \left(2h\omega_1-3A_3\frac{c^2}{c^2+1}-c(8\Omega A_3+\omega_2)+4A_3 \right),
\end{split}
\end{equation*}
and
\begin{equation}\label{derivation-CH-14}
\begin{split}
&\beta-\frac{\alpha}{c}=-c^2 \left(\frac{1}{3}+\beta-\frac{\alpha}{c} \right), \quad \gamma+\frac{3\alpha\,c}{c^2+1}=c^2 \left(\frac{1}{3}+\beta-\frac{\alpha}{c} \right)-cA_4,\\
&  -cA_5=\frac{\delta-\gamma}{2}+\frac{3\alpha\,c}{c^2+1}.
\end{split}
\end{equation}
Therefore, there appears the relations
\begin{equation}\label{derivation-CH-15a}
\begin{split}
&\omega_1 = \frac{-3c(c^2-1)(c^2-2)}{2(1+c^2)^3}, \, \  \omega_2 =\frac{(c^2-2)(c^2-1)^2(8c^2-1)}{2(1+c^2)^5},
\end{split}
\end{equation}
and \begin{equation}\label{derivation-CH-15}
\begin{split}
&\beta-\frac{\alpha}{c}=-\frac{c^2}{3(c^2+1)}, \quad \gamma+\frac{3\alpha\,c}{c^2+1}=-\frac{c^2(5c^2-1)}{3(c^2+1)^3},\\
&  \frac{\delta-\gamma}{2}+\frac{3\alpha\,c}{c^2+1}=-\frac{c^2(3c^4+11c^2+5)}{6(c^2+1)^3},
\end{split}
\end{equation}
which gives a one-parameter expression of $\alpha$, $\beta$, $\gamma$, and $\delta$ with respect to  $p\in \mathbb{R}$
\begin{equation*}
\begin{split}
&\alpha = c\,p, \quad \beta=-\frac{c^2}{3(c^2+1)}+p, \quad \gamma=-\frac{c^2(5c^2-1)}{3(c^2+1)^3}-\frac{3c^2}{c^2+1}p,\\
& \delta=-\frac{c^2(3c^4+16c^2+4)}{3(c^2+1)^3}-\frac{9c^2}{c^2+1}p.
\qedhere
\end{split}
\end{equation*}
\end{proof}

We now generalize Theorem \ref{prop-one-parameter} by replacing the vertically averaged velocity $u$ in \eqref{R-GN-2} with the horizontal velocity
$u^{\theta}$ evaluated at the level line $\theta$ of the fluid domain,
so that $\theta=0$ and $\theta=1$ correspond to the bottom and surface, respectively. The
introduction of $\theta$ allows us to derive an approximation consistent with \eqref{R-GN-2} and build  on
a two-parameter family of equations of the form \eqref{GBBM-1}.

\begin{theorem}\label{prop-two-parameters}
Let $p \in \mathbb{R}$, $\theta \in [0, 1]$ and $\lambda=\frac{1}{2}(\theta^2-\frac{1}{3})$.  Assume that
\begin{equation}\label{derivation-CH-21}
\begin{split}
&\alpha = c\,(p+\lambda), \quad \beta=-\frac{c^2}{3(c^2+1)}+p+\lambda, \quad \gamma=-\frac{c^2(5c^2-1)}{3(c^2+1)^3}-\frac{3c^2}{c^2+1}p,\\
& \delta=-\frac{c^2(3c^4+16c^2+4)}{3(c^2+1)^3}-\frac{3c^2}{c^2+1}(3p+\lambda).
\end{split}
\end{equation}
Then there exists $D>0$ such that:
for all $s \geq 0$ and $T>0$, and for all bounded family $(u^{\varepsilon, \mu, \theta})_{(\varepsilon, \mu)\in \mathcal{P}_{\mu_0, M}} \in L^{\infty}([0, \frac{T}{\varepsilon}], H^{s+D}(\mathbb{R}))$ solving \eqref{GBBM-1},
the family $\{\eta^{\varepsilon, \mu}, u^{\varepsilon, \mu}\}_{(\varepsilon, \mu)\in \mathcal{P}_{\mu_0, M}}$, with (omitting the indexes $\varepsilon, \mu$)
\begin{equation}\label{relation-eta-u-1}
\begin{split}
&u=u^{\theta}+\mu \lambda u^{\theta}_{xx}+\frac{2\lambda}{c}u^{\theta}u^{\theta}_{xx},\quad \eta=F(u^{\theta})+\frac{\lambda}{c} \mu u_{xt}+2\mu\varepsilon \frac{\lambda^2}{c}(1-hc^2)u^{\theta}u^{\theta}_{xx}\,\mbox{with}\\
&F(u) = \frac{1}{c}u- \varepsilon h\, u^2+A_3 \varepsilon^2 \,u^3+ \left(2\Omega A_3+\frac{\omega_2}{4} \right)\varepsilon^3 \,u^4 \\
& \quad \qquad\qquad +\mu \left(\frac{1}{3}+\beta-\frac{\alpha}{c}\right)\, u_{xt}-\varepsilon\mu\bigg(A_1 uu_{xx}+ A_2 u_x^2\bigg)
\end{split}
\end{equation}
is consistent (of order $s$ and on $[0, \frac{T}{\varepsilon}]$) with the R-GN equations \eqref{R-GN-2}.
\end{theorem}
\begin{proof}
Taking  $u=u^{\theta}+\mu \lambda u^{\theta}_{xx}+\frac{2\lambda}{c}u^{\theta}u^{\theta}_{xx}$ in \eqref{GBBM-1} with $\lambda=\frac{1}{2}(\theta^2-\frac{1}{3})$ and $\theta \in [0, 1]$, and $\omega_1$ and $\omega_2$ satisfying \eqref{derivation-CH-14a}, it then follows  from \eqref{derivation-CH-11} that
\begin{equation}\label{derivation-CH-17}
\begin{split}
\eta&=F(u)  := \frac{1}{c}u+\varepsilon\, v =\frac{1}{c}u- \varepsilon h\, u^2+A_3 \varepsilon^2 \,u^3+ \left(2\Omega A_3+\frac{\omega_2}{4} \right)\varepsilon^3 \,u^4 \\
& \quad \qquad\qquad +\mu \left(\frac{1}{3}+\beta-\frac{\alpha}{c} \right)\, u_{xt}-\varepsilon\mu\bigg(A_1 uu_{xx}+ A_2 u_x^2\bigg) +O(\mu\varepsilon^2, \mu^2)\\
&=F(u^{\theta})+\frac{\lambda}{c} \mu u_{xt}+2\mu\varepsilon \frac{\lambda^2}{c}(1-hc^2)u^{\theta}u^{\theta}_{xx}+O(\mu\varepsilon^2, \mu^2).
\end{split}
\end{equation}
Substituting \eqref{derivation-CH-17} into the first equation of \eqref{R-GN-2}, we obtain
\begin{equation}\label{derivation-CH-18}
\begin{split}
&u^{\theta}_t+c u^{\theta}_x+ \frac{3c^2}{c^2+1}\varepsilon u^{\theta}u^{\theta}_x+\omega_1\varepsilon^2 (u^{\theta})^2u^{\theta}_x+\omega_2\varepsilon^3 (u^{\theta})^3u^{\theta}_x\\
&-c^2 \left(\frac{1}{3}+\beta-\frac{\alpha}{c} \right)\mu\, u^{\theta}_{xxt}+\lambda \mu (u^{\theta}_t+c u^{\theta}_x)_{xx}+2\lambda \mu\varepsilon(1+hc^2)(u^{\theta}u^{\theta}_{xx})_x\\
&=\varepsilon\mu \left( \left(c^2 \left(\frac{1}{3}+\beta-\frac{\alpha}{c}\right)-cA_4 \right) u^{\theta}u^{\theta}_{xx}-cA_5 (u^{\theta}_x)^2 \right)_x+O(\mu\varepsilon^2, \mu^2),
\end{split}
\end{equation}
which implies by an iteration argument that $\lambda \mu (u^{\theta}_t+c u^{\theta}_x)_{xx}=-\frac{3c^2}{c^2+1}\lambda \mu \varepsilon \left(u^{\theta}u^{\theta}_{xx}+(u^{\theta}_x)^2 \right)_x+O(\mu\varepsilon^2, \mu^2)$, and then
\begin{equation}\label{derivation-CH-19}
\begin{split}
&u^{\theta}_t+c u^{\theta}_x+ \frac{3c^2}{c^2+1}\varepsilon u^{\theta}u^{\theta}_x+\omega_1\varepsilon^2 (u^{\theta})^2u^{\theta}_x+\omega_2\varepsilon^3 (u^{\theta})^3u^{\theta}_x-c^2 \left(\frac{1}{3}+\beta-\frac{\alpha}{c} \right)\mu\, u^{\theta}_{xxt}\\
&=\varepsilon\mu\left( \left(c^2 \left(\frac{1}{3}+\beta-\frac{\alpha}{c} \right)-cA_4 \right) u^{\theta}u^{\theta}_{xx}+ \left(\frac{3c^2}{c^2+1}\lambda-cA_5 \right) (u^{\theta}_x)^2\right)_x+O(\mu\varepsilon^2, \mu^2).
\end{split}
\end{equation}
Comparing \eqref{derivation-CH-7} with \eqref{derivation-CH-19}, we take $\alpha$, $\beta$, $\gamma$, and $\delta$ in \eqref{GBBM-1} to satisfy the relation
\begin{equation}\label{derivation-CH-20}
\begin{split}
&\beta-\frac{\alpha}{c}=-c^2 \left(\frac{1}{3}+\beta-\frac{\alpha}{c} \right), \quad \gamma+\frac{3\alpha\,c}{c^2+1}=c^2 \left(\frac{1}{3}+\beta-\frac{\alpha}{c} \right)-cA_4,\\
&  \frac{3c^2}{c^2+1}\lambda-cA_5=\frac{\delta-\gamma}{2}+\frac{3\alpha\,c}{c^2+1}.
\end{split}
\end{equation}
Then we deduce  the following two-parameter expression of $\alpha$, $\beta$, $\gamma$, and $\delta$ with respect to the parameter $p\in \mathbb{R}$ and $\lambda=\frac{1}{2}(\theta^2-\frac{1}{3})$ with $\theta \in [0, 1],$ that is,
\begin{equation*}
\begin{split}
&\alpha = c\,(p+\lambda), \quad \beta=-\frac{c^2}{3(c^2+1)}+p+\lambda, \quad \gamma=-\frac{c^2(5c^2-1)}{3(c^2+1)^3}-\frac{3c^2}{c^2+1}p,\\
& \delta=-\frac{c^2(3c^4+16c^2+4)}{3(c^2+1)^3}-\frac{3c^2}{c^2+1}(3p+\lambda).
\end{split}
\end{equation*}
This completes the proof of Theorem  \ref{prop-two-parameters}.
\end{proof}
\begin{remark}\label{rmk-CH-eqns-1}
  In order to get the R-CH equation, it is required that
\begin{equation}\label{derivation-CH-22}
\begin{split}
&-\frac{2c^2}{(c^2+1)}\beta=2\gamma=\delta,
\end{split}
\end{equation}
which gives
\begin{equation}\label{derivation-CH-23}
\begin{split}
 &\lambda+p =\frac{-(c^4+6c^2-1)}{6(c^2+1)^2}, \quad \lambda-p =\frac{c^4+2c^2+2)}{6(c^2+1)^2}.
\end{split}
\end{equation}
Or, what is the same,
\begin{equation}\label{derivation-CH-24}
\begin{split}
 &\lambda=\frac{c^4-2c^2+5)}{12(c^2+1)^2}, \quad p =\frac{-(3c^4+10c^2+3)}{12(c^2+1)^2}.
\end{split}
\end{equation}
Consequently,   the R-CH equation \eqref{R-CH-1} is reestablished. % and \eqref{surface-1}.

%The free surface  $\eta $ with respect to the horizontal component of the velocity $u$ at $ z = z_0 $ under the CH regime $\varepsilon=O(\sqrt{\mu})$ is also given by
%\begin{equation}\label{surface-1}
%\eta = \frac{1}{c} u + \gamma_1\varepsilon u^2 +\gamma_2\varepsilon^2 u^3+\gamma_3\varepsilon^3 u^4+\gamma_4 \varepsilon\mu u_{\xi\xi}+O(\varepsilon^4,\mu^2),
%\end{equation}
%where the constants in the expression are given by
%\begin{equation*}
%\begin{split}
%& \gamma_1=\frac{2-c^2}{2c^2(c^2+1)}, \quad \gamma_2=\frac{(c^2-1)(c^2-2)(2c^2+1)}{2c^3(c^2+1)^3}, \\
%& \gamma_3=-\frac{(c^2-1)^2(c^2-2)(21c^4+16c^2+4)}{8c^4(c^2+1)^5},\\
%& \gamma_4=\frac{z_0^2}{2c}-\frac{3c^2+1}{6c(c^2+1)}=\frac{-(3c^4+6c^2-5)}{12c(c^2+1)^2},
%\end{split}
%\end{equation*}
%here the height parameter $z_0$ is determined by the following form
%\begin{equation}\label{z-0-value}
%z_0^2 = \frac {1}{2} - \frac{2}{3} \frac{1}{(c^2 + 1)} + \frac{4}{3} \frac{1}{(c^2 + 1)^2}.
%\end{equation}
\end{remark}

%%%%%%%%%%%%%%%%%%%%%%%%%%%%%%%%%%%%%%%%%%%%%%%%%%%%%%%%%%%%%%
%%%%%%%%%%%%%%%%%%%%%%%%%%%%%%%%%%%%%%%%%%%%%%%%%%%%%%%%%%%%%
\renewcommand{\theequation}{\thesection.\arabic{equation}}
\setcounter{equation}{0}

%%%%%%%%%%%%%%%%%%%%%%%%%%%%%%%%%%%%%%%%%%%%%%%%%%%%%%%%%%%%%%
%%%%%%%%%%%%%%%%%%%%%%%%%%%%%%%%%%%%%%%%%%%%%%%%%%%%%%%%%%%%%%%%%%

\section{Equation for the surface elevation $ \eta$}\label{Sec_surface}

Proceeding a similar proof in Theorem \ref{prop-one-parameter}, we may derive the following evolution of the surface elevation $ \eta$,
%in the family of equations
\begin{equation}\label{free-GBBM-1}
\begin{split}
&\eta_t+c \eta_x+ \varepsilon \, B\, \eta \eta_x+\bar{\omega}_1\varepsilon^2 \eta^2\eta_x+\bar{\omega}_2\varepsilon^3 \eta^3\eta_x+\mu(\alpha \eta_{xxx}+\beta \eta_{xxt})\\
&\ \ =\varepsilon\mu(\gamma \eta\eta_{xxx}+\delta \eta_x\eta_{xx}),
\end{split}
\end{equation}
which can also be used to construct an approximate solution consistent with the R-GNequations \eqref{R-GN-2}.

To see this, inverting \eqref{eta u} to consider
\begin{equation}\label{u eta}
u := c \eta+\varepsilon \bar{v}
\end{equation}
with $\bar v$ to be determined later, we get from the first equation of\eqref{R-GN-2} that
\begin{equation}\label{free-derivation-CH-1}
\begin{split}
\eta_t+c\eta_x+(\varepsilon\, \bar{v} (1+\varepsilon \eta))_x+\varepsilon (\eta^2)_x=0.
\end{split}
\end{equation}
Plugging \eqref{u eta} into the second equation in \eqref{R-GN-2}, and using the fact $c+2\, \Omega=\frac{1}{c}$, we have
\begin{equation}\label{free-derivation-CH-2}
\begin{split}
&\eta_t+c \eta_x+\varepsilon\,c \,\bar{v}_t+\varepsilon\,\frac{c^3}{2} (\eta^2)_x+\varepsilon\,c^2(\eta\, \varepsilon\, \bar{v})_x +\varepsilon\,\frac{c}{2}\,((\varepsilon\bar{v})^2)_x\\
&=\frac{c^2}{3}\, \mu \, \eta_{txx}+ \mu\, \frac{c}{3}(\varepsilon\bar{v})_{xxt}+\mu\,\varepsilon\bigg(-\frac{c^3}{3}\eta\eta_{xx}
-\frac{c^3}{2}\eta_x^2\bigg)_x+O(\varepsilon^4, \mu\varepsilon^2, \mu^2).
\end{split}
\end{equation}

We now assume that  $\eta$ satisfies  the following generalized BBM equation \eqref{free-GBBM-1}.
It then follows  that up to the $O(\varepsilon^2)$ terms (where $\varepsilon=O(\sqrt{\mu})$)
\begin{equation}\label{free-derivation-CH-5}
\begin{split}
&\eta_t+c \eta_x+ \varepsilon\, B\, \eta\eta_x=O(\varepsilon^2, \mu),
\end{split}
\end{equation}
which also yields that
\begin{equation}\label{free-derivation-CH-5a}
\begin{split}
&\eta_{xt}=-c \eta_{xx}- \varepsilon\, B\, (\eta\eta_{xx}+\eta_x^2)+O(\varepsilon^2, \mu),
\end{split}
\end{equation}
and
\begin{equation}\label{free-derivation-CH-6}
\begin{split}
&\eta_{xxx}=-\frac{1}{c}(\eta_{xxt}+ \varepsilon\, B\,(\eta\eta_x)_{xx})+O(\varepsilon^2, \mu).
\end{split}
\end{equation}
Hence, we can replace the $\eta_{xxx}$ term of \eqref{free-GBBM-1} by this expression to get
\begin{equation}\label{free-derivation-CH-7}
\begin{split}
&\eta_t+c \eta_x+ \varepsilon\, B\,\eta\eta_x+\mu \left(\beta-\frac{\alpha}{c} \right)\eta_{xxt}+\bar{\omega}_1\varepsilon^2 \eta^2\eta_x+\bar{\omega}_2\varepsilon^3 \eta^3\eta_x\\
&=\varepsilon\mu\bigg( \left(\gamma+\frac{\alpha\,B}{c} \right) \eta\eta_{xx}+ \left(\frac{\delta-\gamma}{2}+\frac{\alpha\,B}{c} \right)\eta_x^2\bigg)_x+O(\mu\varepsilon^2, \mu^2).
\end{split}
\end{equation}
Combining \eqref{free-derivation-CH-1} with \eqref{free-derivation-CH-7}, we have
\begin{equation}\label{free-derivation-CH-8}
\begin{split}
&(\varepsilon\, \bar{v} (1+\varepsilon \eta))_x=\bigg(\frac{B}{2}-c\bigg)\varepsilon (\eta^2)_x+\mu \left(\beta-\frac{\alpha}{c} \right)c\,\eta_{txx}+\frac{1}{3}\bar{\omega}_1\varepsilon^2 (\eta^3)_x\\
&+\frac{1}{4}\bar{\omega}_2\varepsilon^3 (\eta^4)_x-\varepsilon\mu\bigg( \left(\gamma+\frac{\alpha\,B}{c} \right) \eta\eta_{xx}+ \left(\frac{\delta-\gamma}{2}+\frac{\alpha\,B}{c}\right)\eta_x^2\bigg)_x+O(\mu\varepsilon^2, \mu^2),
\end{split}
\end{equation}
which implies by integrating \eqref{derivation-CH-8} with respect to $x$ that
\begin{equation}\label{free-derivation-CH-9}
\begin{split}
&\varepsilon\, \bar{v} (1+\varepsilon \eta)=\bigg(\frac{B}{2}-c\bigg)\varepsilon \eta^2+\mu \left(\beta-\frac{\alpha}{c} \right)c\,\eta_{tx}+\frac{1}{3}\bar{\omega}_1\varepsilon^2 \eta^3\\
&+\frac{1}{4}\bar{\omega}_2\varepsilon^3 \eta^4-\varepsilon\mu\bigg( \left(\gamma+\frac{\alpha\,B}{c} \right) \eta\eta_{xx}+ \left(\frac{\delta-\gamma}{2}+\frac{\alpha\,B}{c} \right)\eta_x^2\bigg)+O(\mu\varepsilon^2, \mu^2).
\end{split}
\end{equation}

It thus follows from the Taylor series expansion of $ (1 + \epsilon \eta)^{-1} $ in terms of $ \epsilon $ that
\begin{equation}\label{free-derivation-CH-10}
\begin{split}
&\varepsilon \bar{v} =(1-\varepsilon\eta+\varepsilon^2\eta^2+O(\varepsilon^3))\bigg(\left(\frac{B}{2}-c \right)\varepsilon \eta^2+\mu \left(\beta-\frac{\alpha}{c} \right)c\,\eta_{tx}+\frac{1}{3}\bar{\omega}_1\varepsilon^2 \eta^3\\
&+\frac{1}{4}\bar{\omega}_2\varepsilon^3 \eta^4-\varepsilon\mu\bigg( \left(\gamma+\frac{\alpha\,B}{c} \right) \eta\eta_{xx}+ \left(\frac{\delta-\gamma}{2}+\frac{\alpha\,B}{c} \right)\eta_x^2\bigg)+O(\mu\varepsilon^2, \mu^2)\bigg).
\end{split}
\end{equation}
Or, what is the same,
\begin{equation}\label{free-derivation-CH-11}
\begin{split}
&\varepsilon \bar{v} = \varepsilon \, \left(\frac{B}{2}-c \right)\,\eta^2+\mu \left(\beta-\frac{\alpha}{c} \right)\,\eta_{xt}+ \left(\frac{\bar{\omega}_1}{3}-\frac{B}{2}+c \right) \varepsilon^2 \eta^3 \\
&+ \left(\frac{\bar{\omega}_2}{4}-\frac{\bar{\omega}_1}{3}+\frac{B}{2}-c \right)\varepsilon^3 \eta^4\\
&-\varepsilon\mu\bigg( \left(\gamma+\frac{\alpha\,B}{c}-\beta\, c+\alpha \right) \eta\eta_{xx}+ \left(\frac{\delta-\gamma}{2}+\frac{\alpha\,B}{c} \right)\eta_x^2\bigg)+O(\mu\varepsilon^2, \mu^2).
\end{split}
\end{equation}
Hence it is adduced  from  \eqref{free-derivation-CH-7} that
\begin{equation}\label{free-derivation-CH-12}
\begin{split}
&\varepsilon \, c\, \bar{v}_t = \varepsilon \,(-2c^2) \left(\frac{B}{2}-c \right)\, \eta \eta_x+\varepsilon^2 \,\bar{A}_1\,\eta^2\eta_x+\varepsilon^3 \,\bar{A}_2\,\eta^3\eta_x\\
&+\mu(\alpha\, c-\beta \,c^2)\eta_{xxt} +\varepsilon\mu\bigg(\bar{A}_3 \eta\eta_{xx}+ \bar{A}_4\eta_x^2\bigg)_x +O(\mu\varepsilon^2, \mu^2),
\end{split}
\end{equation}
where
\begin{equation}\label{free-derivation-CH-13}
\begin{split}
\bar{A}_1:= & -2cB \left(\frac{B}{2}-c \right)-3c^2 \left(\frac{\bar{\omega}_1}{3}-\frac{B}{2}+c\right) =-c^2 \bar{\omega}_1+(3c^2-2cB)\left(\frac{B}{2}-c \right),\\
\bar{A}_2:= & -2c\, \bar{\omega}_1 \left(\frac{B}{2}-c \right)-3cB \left(\frac{\bar{\omega}_1}{3}-\frac{B}{2}+c \right)-4c^2 \left(\frac{\bar{\omega}_2}{4}- \left(\frac{\bar{\omega}_1}{3}-\frac{B}{2}+c\right)\right)\\
= & -c^2\,\bar{\omega}_2+ \left(\frac{10}{3}c^2-2cB \right) \bar{\omega}_1 -(4c^2-3cB)\left(\frac{B}{2}-c \right),\\
\bar{A}_3:= &\ 2c^2 \left(\frac{B}{2}-c \right) \left(\beta-\frac{\alpha}{c} \right)+c\,B\,(c\,\beta-\alpha)+c^2\left(\gamma+\frac{\alpha\,B}{c}-\beta\, c+\alpha \right) \\
= &\ (2c^2\,B-3c^3)\left(\beta-\frac{\alpha}{c} \right)+c^2 \left(\gamma+\frac{\alpha\,B}{c} \right), \quad {\rm and} \\
\bar{A}_4:= & -c^2 \left(\frac{B}{2}-c \right) \left(\beta-\frac{\alpha}{c} \right)+c\,B\,(c\,\beta-\alpha)+c^2 \left(\frac{\delta-\gamma}{2}+\frac{\alpha\,B}{c}\right) \\
= &\ \left(\frac{c^2\,B}{2}+c^3 \right) \left(\beta-\frac{\alpha}{c} \right)+c^2 \left(\frac{\delta-\gamma}{2}+\frac{\alpha\,B}{c} \right).
\end{split}
\end{equation}
%\begin{equation}\label{free-derivation-CH-13-a}
%\begin{split}
%&\bar{A}_1:=-c^2 \bar{\omega}_1+(3c^2-2cB)(\frac{B}{2}-c),\\
%&\bar{A}_2:=-c^2\,\bar{\omega}_2+(\frac{10}{3}c^2-2cB) \bar{\omega}_1 -(4c^2-3cB)(\frac{B}{2}-c),\\
%&\bar{A}_3:=(2c^2\,B-3c^3)(\beta-\frac{\alpha}{c})+c^2(\gamma+\frac{\alpha\,B}{c}),\\
%&\bar{A}_4:=(\frac{c^2\,B}{2}+c^3)(\beta-\frac{\alpha}{c})+c^2(\frac{\delta-\gamma}{2}+\frac{\alpha\,B}{c}).
%\end{split}
%\end{equation}
Similarly, we may obtain
\begin{equation}\label{free-derivation-CH-14}
\begin{split}
&\varepsilon\,c^2(\eta\, \varepsilon\, \bar{v})_x =\varepsilon^2\,3c^2\, \left(\frac{B}{2}-c \right)\eta^2\eta_x+\varepsilon^3\,4c^2\, \left(\frac{\bar{\omega}_1}{3}-\frac{B}{2}+c \right)\eta^3\eta_x\\
&\qquad\qquad\qquad-\varepsilon\mu \, c^3 \left(\beta-\frac{\alpha}{c} \right)(\eta\,\eta_{xx})_x+O(\mu\varepsilon^2, \mu^2),\\
&\varepsilon\,\frac{c}{2}\,((\varepsilon\bar{v})^2)_x=\varepsilon^3\, 2c\,\left(\frac{B}{2}-c \right)^2\,\eta^3\, \eta_x+O(\mu\varepsilon^2, \mu^2),\\
&\mu\, \frac{c}{3}(\varepsilon\bar{v})_{xxt}=-\mu\, \varepsilon\,\frac{2c^2}{3} \left(\frac{B}{2}-c \right)(\eta\,\eta_{xx}+\eta_x^2)_x+O(\mu\varepsilon^2, \mu^2).
\end{split}
\end{equation}
Substituting \eqref{free-derivation-CH-12} and \eqref{free-derivation-CH-14} into \eqref{free-derivation-CH-2}, there appears the equation for the function $ \eta $
\begin{equation}\label{free-derivation-CH-15}
\begin{split}
&\eta_t+c \eta_x+\varepsilon\, \left(c^3-2c^2\left(\frac{B}{2}-c\right)\right)\,\eta\eta_x+\varepsilon^2\,\bigg(3c^2\,\left(\frac{B}{2}-c \right)+\bar{A}_1\bigg) \eta^2\eta_x\\
&+\varepsilon^3\,\bigg(4c^2\, \left(\frac{\bar{\omega}_1}{3}-\frac{B}{2}+c\right)+2c\,\left(\frac{B}{2}-c \right)^2+\bar{A}_2\bigg)\eta^3\eta_x-\mu\, c^2 \left(\beta-\frac{\alpha}{c}+\frac{1}{3} \right)\eta_{xxt}\\
&=\mu\,\varepsilon\bigg(\bar{A}_5\,\eta\eta_{xx}+\bar{A}_6 \,\eta_x^2\bigg)_x+O(\varepsilon^4, \mu\varepsilon^2, \mu^2),
\end{split}
\end{equation}
where
\begin{equation}\label{free-derivation-CH-16}
\begin{split}
&\bar{A}_5:=-\frac{c^3}{3}-4c^2 \left(\frac{B}{2}-c \right) \left(\beta-\frac{\alpha}{c}-\frac{1}{6} \right)-c^2 \left(\gamma+\frac{\alpha\,B}{c} \right) \quad {\rm and} \\
& \bar{A}_6:=-\frac{c^3}{2}+\frac{2c^2}{3} \left(\frac{B}{2}-c \right)- \left(\frac{c^2\,B}{2}+c^3 \right)\left(\beta-\frac{\alpha}{c}\right)-c^2 \left(\frac{\delta-\gamma}{2}+\frac{\alpha\,B}{c}\right).
\end{split}
\end{equation}
On account of  \eqref{free-derivation-CH-7} and \eqref{free-derivation-CH-16}, we take $\bar{\omega}_1$, $\bar{\omega}_2$, $\alpha$, $\beta$, $\gamma$, and $\delta$ in \eqref{free-GBBM-1} to satisfy the relation
\begin{equation}\label{free-derivation-CH-19}
\begin{split}
&B=c^3-2c^2 \left(\frac{B}{2}-c\right),\quad \bar{\omega}_1=3c^2\,\left(\frac{B}{2}-c\right)+\bar{A}_1,\\
&\bar{\omega}_2=4c^2\, \left(\frac{\bar{\omega}_1}{3}-\frac{B}{2}+c \right)+2c\,\left(\frac{B}{2}-c \right)^2+\bar{A}_2,\\
&\beta-\frac{\alpha}{c}=- c^2 \left(\beta-\frac{\alpha}{c}+\frac{1}{3}\right), \quad \gamma+\frac{\alpha\,B}{c}=\bar{A}_5, \quad \frac{\delta-\gamma}{2}+\frac{\alpha\,B}{c}=\bar{A}_6.
\end{split}
\end{equation}
It thus transpires that
\begin{equation}\label{free-derivation-CH-20}
\begin{split}
&B=\frac{3c^3}{1+c^2}, \quad \bar{\omega}_1 = \frac{-3c^3(2-c^2)}{(1+c^2)^3}, \quad  \bar{\omega}_2 =\frac{c^3(2-c^2)(c^6+9c^4-7c^2+3)}{(1+c^2)^5},
\end{split}
\end{equation}
and \begin{equation}\label{free-derivation-CH-21}
\begin{split}
&\beta-\frac{\alpha}{c}=-\frac{c^2}{3(c^2+1)}, \quad \gamma+\frac{3\,c^2}{c^2+1}\alpha=-\frac{c^3(-2c^4+7c^2+3)}{3(c^2+1)^3},\\
&  \frac{\delta-\gamma}{2}+\frac{3\,c^2}{c^2+1}\alpha=-\frac{c^3(-4c^4+6c^2+7)}{6(c^2+1)^3},
\end{split}
\end{equation}
which implies an one-parameter expression of $\alpha$, $\beta$, $\gamma$, and $\delta$ with respect to the parameter $p\in \mathbb{R}$
\begin{equation}\label{free-derivation-CH-22}
\begin{split}
&\alpha = c\,p, \quad \beta=-\frac{c^2}{3(c^2+1)}+p, \quad \gamma=-\frac{c^3(-2c^4+7c^2+3)}{3(c^2+1)^3}-\frac{3c^3}{c^2+1}p,\\
& \delta=-\frac{c^3(-6c^4+13c^2+10)}{3(c^2+1)^3}-\frac{9c^3}{c^2+1}p.
\end{split}
\end{equation}
Consequently, we have
\begin{equation}\label{free-derivation-CH-23}
\begin{split}
&\varepsilon \bar{v} = \varepsilon \, \frac{c(c^2-2)}{2(c^2+1)}\,\eta^2-\mu\,\frac{c^2}{3(1+c^2)}\,\eta_{xt}+ \varepsilon^2\,\frac{c(2-c^2)(c^4+1)}{2(1+c^2)^3} \,  \eta^3 \\
&-\varepsilon^3\,\frac{c(2-c^2)(c^8-5c^6+11c^4+c^2+2)}{4(1+c^2)^5} \,  \eta^4\\
&+\varepsilon\mu\bigg(\frac{c^3(-3c^4+5c^2+2)}{3(1+c^2)^3} \, \eta\eta_{xx}+ \frac{c^3(-4c^4+6c^2+7)}{6(c^2+1)^3}\eta_x^2\bigg)+O(\mu\varepsilon^2, \mu^2).
\end{split}
\end{equation}
In view of the proof of Theorem \ref{prop-one-parameter}, we have the following consistent result for the free surface $ \eta. $

\begin{theorem}\label{thm-free-one-parameter}
Let $p \in \mathbb{R}$. Assume that \eqref{free-derivation-CH-20} and \eqref{free-derivation-CH-22} hold.
Then there exists $D>0$ such that:
for all $s \geq 0$ and $T>0$, and for all bounded family $(\eta^{\varepsilon, \mu})_{(\varepsilon, \mu)\in \mathcal{P}_{\mu_0, M}} \in L^{\infty}([0, \frac{T}{\varepsilon}], H^{s+D}(\mathbb{R}))$ solving \eqref{free-GBBM-1},
the family $\{\eta^{\varepsilon, \mu}, u^{\varepsilon, \mu}\}_{(\varepsilon, \mu)\in \mathcal{P}_{\mu_0, M}}$, with (omitting the indexes $\varepsilon, \mu$)
\begin{equation}\label{free-relation-eta-u-1}
\begin{split}
&u=c\, \eta+  \varepsilon \, \frac{c(c^2-2)}{2(c^2+1)}\,\eta^2-\mu\,\frac{c^2}{3(1+c^2)}\,\eta_{xt}+ \varepsilon^2\,\frac{c(2-c^2)(c^4+1)}{2(1+c^2)^3} \,  \eta^3 \\
&\quad-\varepsilon^3\,\frac{c(2-c^2)(c^8-5c^6+11c^4+c^2+2)}{4(1+c^2)^5} \,  \eta^4\\
&\quad +\varepsilon\mu\bigg(\frac{c^3(-3c^4+5c^2+2)}{3(1+c^2)^3} \, \eta\eta_{xx}+ \frac{c^3(-4c^4+6c^2+7)}{6(c^2+1)^3}\eta_x^2\bigg)
\end{split}
\end{equation}
is consistent (of order $s$ and on $[0, \frac{T}{\varepsilon}]$) with R-GN equations \eqref{R-GN-2}.
\end{theorem}

\begin{remark}\label{rmk-free-CH-eqns-1}
 Choosing $p=\frac{2c^4+c^2-4}{9(1+c^2)^2}$, we have
 \begin{equation}\label{free-derivation-CH-24}
\begin{split}
&\alpha = \frac{c(2c^4+c^2-4)}{9(1+c^2)^2}, \quad \beta=-\frac{c^4+2c^2+4}{9(c^2+1)^2}.\\
& \delta=2\gamma=-\frac{2c^3(8c^2-1)}{3(c^2+1)^3},
\end{split}
\end{equation}
Equation \eqref{free-GBBM-1} then reads
\begin{equation}\label{free-GBBM-2}
\begin{split}
&\eta_t+c \eta_x+ \varepsilon \, B\, \eta \eta_x+\bar{\omega}_1\varepsilon^2 \eta^2\eta_x+\bar{\omega}_2\varepsilon^3 \eta^3\eta_x+\mu(\alpha \eta_{xxx}+\beta \eta_{xxt})\\
&=-\frac{2c^3(8c^2-1)}{3(c^2+1)^3}\, \varepsilon\mu\, (2 \eta_x\eta_{xx}+\eta\eta_{xxx}).
\end{split}
\end{equation}
%where the ratio between the coefficients of $\eta_x\eta_{xx}$ and $\eta\eta_{xxx}$ is $2:1$.

\end{remark}

%%%%%%%%%%%%%%%%%%%%%%%%%%%%%%%%%%%%%%%%%%%%%%%%%%%%%%%%%%%%%%
%%%%%%%%%%%%%%%%%%%%%%%%%%%%%%%%%%%%%%%%%%%%%%%%%%%%%%%%%%%%%
\renewcommand{\theequation}{\thesection.\arabic{equation}}
\setcounter{equation}{0}

%%%%%%%%%%%%%%%%%%%%%%%%%%%%%%%%%%%%%%%%%%%%%%%%%%%%%%%%%%%%%%
%%%%%%%%%%%%%%%%%%%%%%%%%%%%%%%%%%%%%%%%%%%%%%%%%%%%%%%%%%%%%%%%%%

\section{Local well-posedness}\label{Sec_local}

In this section, we investigate local well-posedness of the Cauchy problem to the R-CH equation in \eqref{R-CH-1} and the R-GN equations in \eqref{R-GN-1}.

Let
\begin{equation*}
\|u\|^2_{H^{s+1}_{\mu}} = \|u\|^2_{H^s} + \mu \|\partial_x u\|^2_{H^s}.
\end{equation*}
For some $\mu_0 > 0$ and $M > 0$, we define the Camassa-Holm regime $\mathcal{P}_{\mu_0, M} := \{(\varepsilon, \mu): <\mu \leq \mu_0, 0<\varepsilon \leq M \sqrt{\mu}\}$. Then, we have the following result.
\begin{theorem}\label{thm-r-CH-local}
Let $u_0 \in H^{s+1}_{\mu}(\mathbb{R})$, $\beta>0$, $\mu_0 > 0$ and $M > 0$, $s > \frac{3}{2}$. Then, there exist $T > 0$ and a unique family of solutions $\left(u_{\varepsilon,\mu}\right)|_{(\varepsilon,\mu) \in \mathcal{P}_{\mu_0, M}}$  in $C\left(\left[0,\frac{T}{\varepsilon}\right]; H^{s+1}_{\mu}(\mathbb{R})\right) \cap C^1\left(\left[0,\frac{T}{\varepsilon}\right] ;H^s_{\mu}(\mathbb{R})\right)$ to the Cauchy problem
\begin{equation}\label{r-CH-eqns-1}
\begin{cases}
&\partial_t u-\beta\mu \partial_t u_{xx} + c u_x + 3\alpha\varepsilon uu_x - \beta_0\mu u_{xxx} + \omega_1 \varepsilon^2u^2u_x + \omega_2 \varepsilon^3u^3u_x   \\
&\qquad\qquad \qquad\qquad \qquad\qquad \qquad\qquad \qquad= \alpha\beta\varepsilon\mu( 2u_{x}u_{xx}+uu_{xxx}),\\
&u|_{t = 0} = u_0.
\end{cases}
\end{equation}
\end{theorem} We omit the proof of this result, since  it is similar to that in \cite{CL09}.

Attention is now turned to  the case of  local well-posedness of the R-GN equations in \eqref{R-GN-2}. Note that we may rewrite \eqref{R-GN-2} as
\begin{equation}\label{R-GN-1}
\begin{cases}
\eta_t + (h\, u)_x = 0,\\
\mathfrak{T}(u_t +  \varepsilon u u_x )+h \eta_x- 2 \Omega (h^2u_x+ \varepsilon h u\eta_x)+\varepsilon \mu \,h\,Q[h]u = 0,
\end{cases}
\end{equation}
where $h := 1+\varepsilon \eta$, $\mathfrak{T}=\mathfrak{T}[h]  := h\,  -\frac{\mu}{3}\partial_x(h^3 \partial_x \, )$, $Q[h]f := \frac{2}{3h}\partial_x(h^3 f_x^2)$.

If we denote
\begin{equation*}
\begin{split}
& U := \left(
\begin{array}{lll}
\eta\\
u
\end{array}\right), \quad Q_1[U]f := \frac{2}{3}\varepsilon\mu \partial_x(h^3u_x \,f)-2\Omega h^2 f,\\
& A(U) :=
\left(
\begin{array}{lll}
\varepsilon u & h\\
 \mathfrak{T}^{-1}((1-2\varepsilon\Omega u)h\cdot)& \varepsilon\, u+ \mathfrak{T}^{-1}Q_1[U]
\end{array}\right),
\end{split}
\end{equation*}
then we may rewrite \eqref{R-GN-1} as the following hyperbolic equations:
\begin{equation}\label{R-GN-3}
\begin{cases}
&\partial_t U+ A[U] \partial_x U=0,\\
&U|_{t=0}=U_0.
\end{cases}
\end{equation}

%A symmetrizer for $A[U]$ is given by
%\begin{equation}\label{symmetrizer-1}
%\begin{split}
%& S :=
%\left(
%\begin{array}{lll}
%1 - 2\varepsilon\Omega u & 0\\
%0&  \mathfrak{T}
%\end{array}\right). %,\quad \mbox{where} \quad \mathfrak{T}_1  :=  {\kappa}_{\Omega} \mathfrak{T}, \quad {\kappa}_{\Omega}={\kappa}_{\Omega}(u) := (1-2\varepsilon\Omega u)^{-1}.
%\end{split}
%\end{equation}

From the structure of \eqref{R-GN-3}, it is convenient to introduce the following  energy spaces $X^{s}$.
\begin{definition}
  Given $s \geq 0$, $T>0$. We define the Hilbert space
\begin{equation*}
X^{s}=X^{s}(\mathbb{R}):=\{(\eta, u)^T\in ( H^s( \mathbb{R}))^2 :\, \|\eta\|_{H^s}^2+ \|u\|_{H^s}^2+ \mu \|\partial_x u\|_{H^s}^2 <+\infty\}
\end{equation*}
equipped with the norm $\|(\eta, u)^T \|_{X^{s}(\mathbb{R})}:= (\|\eta\|_{H^s}^2+ \|u\|_{H^s}^2+ \mu \|\partial_x u\|_{H^s}^2)^{\frac{1}{2}}$ for any $(\eta, u)^T \in X^{s}(\mathbb{R})$ and the canonical inner product, while $X^{s}_T$ stands for  $C([0, \frac{T}{\varepsilon}], X^s)$.
\end{definition}

The main result of this section is the following.
\begin{theorem}\label{thm-local}
For $s >\frac{3}{2}$, let the initial data
$ U_0 := (\eta_0, u_0)^T \in X^{s}(\mathbb{R})$ satisfy the following condition:
\begin{equation*}%\label{initial-condition-1}
\mbox{there is a constant} \,\,b_0 > 0 \,\mbox{such that} \,\min \left\{\inf_{x\in \mathbb{R}} (1+\varepsilon\eta_0), \inf_{x\in \mathbb{R}} (1-2\Omega \varepsilon u_0) \right\} \geq b_0.
\end{equation*}
Then there exists a positive maximal existence time $T_{max}>0$, uniformly bounded from below with respect to $\varepsilon, \, \mu \in (0, 1)$, such that the Green-Naghdi equations \eqref{R-GN-3} admit a unique solution $U = (\eta , u)^T \in X^s_{T_{max}}$ preserving the condition \eqref{initial-condition-1} for any $ t \in [0, \frac{T_{max}}{\varepsilon})$. In particular if $T_{max} < +\infty$, there holds
\begin{equation}\label{blowup-cond-1}
\|U(t, \cdot)\|_{X^s} \rightarrow +\infty \quad \mbox{as} \quad t \rightarrow \frac{T_{max}}{\varepsilon},
\end{equation}
or
\begin{equation}\label{blowup-cond-2}
\min \left\{\inf_{x\in \mathbb{R}} (1+\varepsilon\eta), \inf_{x\in \mathbb{R}} (1-2\Omega \varepsilon u) \right\} \rightarrow 0 \quad \mbox{as} \quad t \rightarrow \frac{T_{max}}{\varepsilon}.
\end{equation}
Moreover, the energy
\begin{equation*}
E(\eta, u) := \|\eta\|_{L^2}^2+(\mathfrak{T}u, u)=\int_{\mathbb{R}}(\eta^2+h\, u^2+\frac{\mu}{3}h^3u_x^2)\,dx
\end{equation*}
is independent of the existence time $t>0$.
\end{theorem}

\subsection{Symmetrization}
It is easy to see that a symmetrizer for \eqref{R-GN-3} can be found to be
\begin{equation}\label{symmetrizer-1}
\begin{split}
& S :=
\left(
\begin{array}{lll}
1 & 0\\
0&  \mathfrak{T}_1.
\end{array}\right),\quad \mbox{where} \quad \mathfrak{T}_1  :=  {\kappa}_{\Omega} \mathfrak{T}, \quad {\kappa}_{\Omega}={\kappa}_{\Omega}(u) := (1-2\varepsilon\Omega u)^{-1}.
\end{split}
\end{equation}
Thus a natural energy for the equation \eqref{R-GN-3} is given by
\begin{equation}\label{energy-equiv-1}
E^s(U)^2 :=(\Lambda^s U, S\, \Lambda^s U).
\end{equation}
The link between $E^s(U)$ and the $X^s$-norm is investigated in the following lemma, which was obtained in \cite{Is2011} (up to a slight modification).

\begin{lemma}[\cite{Is2011}]\label{lem-equiv-norm}
Let $s \geq 0$, $u \in L^{\infty}(\mathbb R)$, and $\eta \in W^{1,\infty}(\mathbb R)$ be such that
\begin{equation}\label{initial-condition-1}
\mbox{there is a constant} \,\,b_0 > 0 \,\mbox{such that} \,\min \left\{\inf_{x\in \mathbb{R}} (1+\varepsilon\eta), \inf_{x\in \mathbb{R}} (1-2\Omega \varepsilon u) \right\} \geq b_0.
\end{equation}
Then $E^s(U)$ is uniformly equivalent to the $X^s$-norm
with respect to $(\mu,\, \varepsilon) \in (0, \frac{1}{2}) \times (0, \frac{1}{2})$:
\begin{equation*}
E^s(U) \leq  C(\|u\|_{L^{\infty}}, \|h\|_{L^{\infty}}, \|h_x\|_{L^{\infty}})\|U\|_{X^s} ,
\end{equation*}
and
\begin{equation*}
\|U\|_{X^s} \leq C \left(\frac{1}{b_0} \right) E^s(U).
\end{equation*}
\end{lemma}

%Now we are in a position to state the local well-posedness of the Cauchy problem \eqref{R-GN-3}.

\begin{remark}\label{rmk-initial-con-1}
According to the definition of the matrix operator $S$ in the natural energy $E^s(U)$, it seems reasonable to restrict the additional initial condition \eqref{initial-condition-1} in Theorem \ref{thm-local} if we want to avoid   the degenerate case.
\end{remark}
In order to prove Theorem \ref{thm-local}, we first recall some fundamental properties of the pseudo-differential operators.
\begin{lemma}[Commutator estimates \cite{Ka-Pon-88}]\label{lemma-comm}
Let $\Lambda^s:=(1-\partial_x^2)^{\frac{s}{2}}$ with $s>0$. Then the  following two estimates are true:

(i) $\|[\Lambda^s, f]g\|_{L^2(\mathbb R)} \leq C(\|f\|_{H^s}\|g\|_{L^\infty(\mathbb R)}+\|f_x\|_{L^\infty(\mathbb R)}\|g\|_{H^{s-1}(\mathbb R)})$;

(ii) $\|[\Lambda^s, f]g\|_{L^2(\mathbb R)} \leq C\|f_x\|_{H^{q_0}(\mathbb R)}\|g\|_{H^{s-1}(\mathbb R)}$, \quad $\forall\, 0 \leq s \leq q_0+1, \, q_0>\frac{1}{2}$,

\noindent where all the constants $C$ are independent of $f$ and $g$.
\end{lemma}
The following two lemmas provide invertibility of $\mathfrak T$ and the estimates for $\mathfrak{T}^{-1}$. The proof follows the argument in \cite{Is2011}. %, the following result about the operator $\mathfrak{T}_1$ was obtained (with a slight modification).
\begin{lemma}[\cite{Is2011}]\label{lem-operator-1}
Let $u\in L^{\infty}(\mathbb{R})$ and $\eta \in W^{1, \infty}(\mathbb{R})$ be such that \eqref{initial-condition-1} holds. Then the operator
\begin{equation}\label{operator-1}
\mathfrak{T}: \, H^2(\mathbb{R})\rightarrow L^2(\mathbb{R})
\end{equation}
is well defined, one-to-one and onto.
\end{lemma}
%The following lemma then gives some properties of the inverse operator $\mathfrak{T}$.
\begin{lemma}[\cite{Is2011}]\label{lem-operator-2}
Let $s_0 > 1/2 $ and $\eta \in H^{1+s_0}(\mathbb{R})$ be such that \eqref{initial-condition-1} holds. Then

(i) $\forall \, 0 \leq s \leq 1+s_0$, $\|\mathfrak{T}^{-1} f \|_{H^s} + \sqrt{\mu}\|\partial_x \mathfrak{T}^{-1} f \|_{H^s} \leq C(\frac{1}{b_0}, \|h-1\|_{H^{1+s_0}} )\|f\|_{H^s}$,

(ii) $\forall \, 0 \leq s \leq 1+s_0$, $\sqrt{\mu}\|\mathfrak{T}^{-1}\partial_x  g \|_{H^s}\leq C(\frac{1}{b_0}, \|h-1\|_{H^{1+s_0}} )\|g\|_{H^s}$ , \quad {\rm and}

(iii) If $s \geq 1+s_0$ and $\eta \in H^{s}(\mathbb{R})$, then:
$\|\mathfrak{T}^{-1}\|_{H^s\rightarrow H^s}+\sqrt{\mu}\|\mathfrak{T}^{-1}\partial_x  \|_{H^s\rightarrow H^s} \leq C_s$,
where $C_s$ is a constant depending on $\frac{1}{b_0}, \,\|h-1\|_{H^{s}}$ and independent of $(\varepsilon, \mu)\in (0, 1/2)$.
\end{lemma}

We further have the following results.
\begin{lemma}\label{lem-operator-diff-1}
Let $s> 3/2 $, $f\in X^{s}$,  and $U_1, \, U_2 \in X^{s}(\mathbb{R})$ be such that \eqref{initial-condition-1} holds. Then we have
\begin{equation}\label{diffe-opera-1}
\begin{split}
& \left \|  \left((\mathfrak{T}[h_1] )^{-1}-(\mathfrak{T}[h_2] \right)^{-1}) f \right\|_{H^{s-1}}  \\
& \quad \leq C \left(\frac{1}{b_0}, \|h_1-1\|_{H^{s}}, \|h_2-1\|_{H^{s}} \right)\|f\|_{H^{s-1}}\|h_1-h_2\|_{H^{s-1}}
\end{split}
\end{equation}
and
\begin{equation}\label{diffe-opera-0}
\begin{split}
&\|(A[U_1]-A[U_2]) f \|_{H^{s-1}}  \\
& \quad \leq \varepsilon\, C \left(\frac{1}{b_0}, \|h_1-1\|_{H^{s}}, \|h_2-1\|_{H^{s}}, \varepsilon\|(u_1, \, u_2)\|_{H^{s}} \right)\|f\|_{H^{s-1}}\|U_1-U_2\|_{H^{s-1}}.
\end{split}
\end{equation}
\end{lemma}
\begin{proof}
Denote $g=g_1-g_2:=(\mathfrak{T}[h_1])^{-1}f-(\mathfrak{T}[h_2])^{-1} f$, we have $f=\mathfrak{T}[h_1]g_1=\mathfrak{T}[h_2]g_2$, which implies
\begin{equation}\label{diffe-opera-2}
0= \mathfrak{T}[h_1]g_1-\mathfrak{T}[h_2]g_2=\mathfrak{T}[h_1](g_1-g_2)+(\mathfrak{T}[h_1]-\mathfrak{T}[h_2])g_2.
\end{equation}
There then  appears that
\begin{equation}\label{diffe-opera-3}
\begin{split}
 &g_1-g_2= -(\mathfrak{T}[h_1])^{-1}\bigg((\mathfrak{T}[h_1]-\mathfrak{T}[h_2])g_2\bigg)\\
&=-(\mathfrak{T}[h_1])^{-1}\bigg((h_1-h_2)g_2-\frac{\mu}{3} \partial_x((h_1-h_2)(h_1^2+h_1h_2+h_2^2)\partial_xg_2)\bigg).
\end{split}
\end{equation}
Therefore, thanks to Lemma \ref{lem-operator-2}, we have
\begin{equation}\label{diffe-opera-4}
\begin{split}
 &\|g_1-g_2\|_{H^{s-1}}\leq C \left(\frac{1}{b_0}, \|h_1-1\|_{H^{s}} \right)\\
 &\qquad\qquad\times \bigg(\|(h_1-h_2)g_2\|_{H^{s-1}}  + \sqrt{\mu}\|(h_1-h_2)(h_1^2+h_1h_2+h_2^2)\partial_xg_2\|_{H^{s-1}} \bigg)\\
 &\leq C \left(\frac{1}{b_0}, \|h_1-1\|_{H^{s}} \right)\|h_1-h_2\|_{H^{s-1}} \left(1+\|h_1-1\|_{H^{s-1}}^2+\|h_2-1\|_{H^{s-1}}^2 \right)\\
 &\qquad\qquad\times \bigg(\|g_2\|_{H^{s-1}}  + \sqrt{\mu}\|\partial_xg_2\|_{H^{s-1}} \bigg).
\end{split}
\end{equation}
Applying Lemma \ref{lem-operator-2} again, we infer
\begin{equation}\label{diffe-opera-5}
\begin{split}
 &\|g_2\|_{H^{s-1}} = \|(\mathfrak{T}[h_2])^{-1}f\|_{H^{s-1}} \leq C \left(\frac{1}{b_0}, \|h_2-1\|_{H^{s}} \right) \|f\|_{H^{s-1}},\\
 &\sqrt{\mu}\|\partial_xg_2\|_{H^{s-1}} \leq \sqrt{\mu}\|\partial_x(\mathfrak{T}[h_2])^{-1}f\|_{H^{s-1}}\leq C \left(\frac{1}{b_0}, \|h_2-1\|_{H^{s}} \right) \|f\|_{H^{s-1}},
\end{split}
\end{equation}
which, along with \eqref{diffe-opera-4}, implies \eqref{diffe-opera-1}.

A similar argument as above leads to \eqref{diffe-opera-0}.
\end{proof}

\subsection{Linearized analysis}
In order to prove Theorem \ref{thm-local}, we first establish the existence, uniqueness and regularity for solutions to the following linearized system of \eqref{R-GN-3}:
\begin{equation}\label{R-GN-linearized-1}
    \begin{cases}
        &\partial_t{U} + A[\underline{U}] \partial_x{U} =\varepsilon F; \\
        & U|_{t = 0} = U_0,
    \end{cases}
\end{equation}
where $F=(F_1, F_2)^{T}$, $ \underline{U} = (\underline{\eta},\  \underline{u})^{T} \in  X_{T}^{s}$ is such that $\partial_t{\underline{U}} \in X_{T}^{s - 1}$ and satisfies the condition (3) on $[0,\ \frac{T}{\varepsilon}]$.

\begin{lemma}\label{thm-linearized-1}
Let $s> \frac{3}{2}$. Assume that $\underline{U}
=(\underline{\eta}, \underline{u})^T\in X^s_T$  such that $\partial_t\underline{U}\in X^{s-1}_T$ and  the condition
\eqref{initial-condition-1}is satisfied  on $[0, \frac{T}{\varepsilon}]$. Then for all $U_0\in X^s$ there exists a unique solution $U=(\eta, u)^T\in X^s_T$
to \eqref{R-GN-linearized-1} and for all $0\leq t\leq\frac{T}{\varepsilon}$
\begin{equation*}
E^s(U(t))\leq e^{\varepsilon\lambda_T\,t}E^s(U_0)+C\, \varepsilon \,\int_0^{t}e^{\varepsilon\lambda_T(t-\tau)}\|F(\tau)\|_{X^s}^2\, d\tau,
\end{equation*}
for some $\lambda_T=\lambda_T(sup_{0\leq t\leq T/\varepsilon}E^s(\underline{U}(t)), sup_{0\leq t\leq T/\varepsilon}
\|(\partial_t\underline{\eta}(t), \partial_t\underline{u}(t), \sqrt{\mu}\underline{u}_{tx}(t))|_{L^\infty})$.
\end{lemma}
\begin{proof}
Using a smooth approximation argument and its uniform estimates, we may establish the existence part of Lemma \ref{thm-linearized-1}, and then applying Gronwall's lemma to achieve the uniqueness part. All the processes of the proof depend essentially on sone necessary {\it a priori} estimates. For the sake of simplicity, we just focus our attention on the proof of the uniform energy estimates.

In fact, for any $\lambda\in\mathbb{R}$ (to be determined in the last step), we have
\begin{equation}\label{linearized-est-1}
e^{\varepsilon\lambda t}\frac{d}{dt}(e^{-\varepsilon\lambda t}E^s(U)^2)=-\varepsilon\lambda E^s(U)^2
+\frac{d}{dt} (E^s(U)^2).
\end{equation}
By the definition of $E^s(U)^2$ and the identity
\begin{equation}\label{linearized-est-2}
(f, \, \underline{\mathfrak{T}}_1\,g)=(\underline{\mathfrak{T}}_1 \,f, \,g)+\frac{\mu}{3}\bigg(f, \, 2[\underline{\kappa}_{\Omega}]_x \underline{h}^3g_x+([\underline{\kappa}_{\Omega}]_{x}h^3)_xg\bigg),
\end{equation}
we get
\begin{equation}\label{linearized-est-3}
\begin{split}
&\frac{d}{dt}(E^s(U)^2)=2(\Lambda^s\eta, \Lambda^s\eta_t)+(\Lambda^su_t, \underline{\mathfrak{T}_1}\Lambda^su)+(\Lambda^su, \underline{\mathfrak{T}_1}\Lambda^su_t)
+(\Lambda^su, [\partial_t, \underline{\mathfrak{T}_1}]\Lambda^su)\\
&=2(\Lambda^s\eta, \Lambda^s\eta_t)+2(\Lambda^su, \underline{\mathfrak{T}_1}\Lambda^su_t)
+(\Lambda^su, [\partial_t, \underline{\mathfrak{T}_1}]\Lambda^su)+\frac{\mu}{3}\bigg(\Lambda^su_t, \, 2[\underline{\kappa}_{\Omega}]_x \underline{h}^3\Lambda^su_x\bigg)\\
&+\frac{\mu}{3}\bigg(\Lambda^su_t, \, ([\underline{\kappa}_{\Omega}]_{x}h^3)_x\Lambda^su\bigg),
\end{split}
\end{equation}
which along with the equation \eqref{R-GN-linearized-1} implies
\begin{equation}\label{energy-est-1}
\begin{split}
&\frac{1}{2}e^{\varepsilon\lambda t} \frac{d}{dt}(e^{-\varepsilon\lambda t}E^s(U)^2)=
-\frac{\varepsilon\lambda}{2}E^s(U)^2-(SA[\underline{U}]\Lambda^s\partial_xU, \Lambda^sU)\\
&-([\Lambda^s, A[\underline{U}]]
\partial_xU, S\Lambda^sU)+2\varepsilon(\Lambda^sU, A[\underline{U}]]\Lambda^sF)\\
&+\frac{1}{2}(\Lambda^su, [\partial_t, \underline{\mathfrak{T}_1}]\Lambda^su)+\frac{\mu}{6}\bigg(\Lambda^su_t, \, ([\underline{\kappa}_{\Omega}]_{x}h^3)_x\Lambda^su+ 2[\underline{\kappa}_{\Omega}]_x \underline{h}^3\Lambda^su_x\bigg).
\end{split}
\end{equation}
Thanks to the expression of the equation \eqref{R-GN-linearized-1} and Lemma \ref{lem-operator-2}, we have
\begin{equation*}
\begin{split}
&\|u_t\|_{H^s}=\|\mathfrak{T}^{-1}(\underline{\kappa}_{\Omega}^{-1}\underline{h}\eta_x )+ \varepsilon \underline{u} u_x+\frac{2}{3} \varepsilon \mu  \mathfrak{T}^{-1} \partial_x (\underline{h}^3 \underline{u}_x u_x)-2\Omega \mathfrak{T}^{-1}(\underline{h}^2 u_x) +\varepsilon F\|_{H^s}\\
& \quad \leq C\left(\frac{1}{b_0}, \,\|\underline{h}-1\|_{H^{s}} \right)\\
&\qquad\qquad\qquad \times (1+\varepsilon\|\underline{u}\|_{H^s}+\sqrt{\mu}\|\underline{u}_x\|_{H^s}) (1+\varepsilon\|\underline{\eta}\|_{H^s})^3(\|u_x\|_{H^s}+\|\eta_x\|_{H^s})+\varepsilon \|F\|_{H^s}.
\end{split}
\end{equation*}
On the other hand, it is found that
\begin{equation*}
|2\varepsilon(\Lambda^sU, A[\underline{U}]]\Lambda^sF)|\leq C\, \varepsilon \,E^s(U)\|F\|_{X^s},
\end{equation*}
and
\begin{equation*}
\begin{split}
&\|([\underline{\kappa}_{\Omega}]_{x}h^3)_x\Lambda^su+ 2[\underline{\kappa}_{\Omega}]_x \underline{h}^3\Lambda^su_x\|_{L^2}\\
& \quad \leq C \varepsilon (1+\varepsilon\|\underline{\eta}\|_{L^{\infty}})^3(\|\underline{u}_{xx}\|_{L^{\infty}}
+\varepsilon\|\underline{u}_{x}\|_{L^{\infty}}^2+\varepsilon\|\underline{u}_{x}\|_{L^{\infty}}\|\underline{u}_{x}\|_{L^{\infty}})\|u\|_{H^s}\\
&\qquad +C \varepsilon(1+\varepsilon\|\underline{\eta}\|_{L^{\infty}})^3 \|\underline{u}_{x}\|_{L^{\infty}} \|u_x\|_{H^s}.
\end{split}
\end{equation*}
It then follows that
\begin{equation*}
\begin{split}
&|\frac{\mu}{6}(\Lambda^su_t, \, ([\underline{\kappa}_{\Omega}]_{x}h^3)_x\Lambda^su+ 2[\underline{\kappa}_{\Omega}]_x \underline{h}^3\Lambda^su_x)|\\
& \quad \leq \frac{\mu}{6}\|u_t\|_{H^s}\|([\underline{\kappa}_{\Omega}]_{x}h^3)_x\Lambda^su+ 2[\underline{\kappa}_{\Omega}]_x \underline{h}^3\Lambda^su_x\|_{L^2}
\leq \varepsilon C(E^s(\underline{U}))E^s(U)^2 .
\end{split}
\end{equation*}
In order to estimate $(SA[\underline{U}]\Lambda^s\partial_xU, \Lambda^sU)$, we first deduce from
\begin{equation}\label{matrix-ope-1}
SA[\underline{U}]=\left(
                    \begin{array}{cc}
                      \varepsilon\underline{u} & \underline{h} \\
                      \underline{h} & \underline{\mathfrak{T}_1}(\varepsilon\underline{u}\cdot)+\underline{\kappa}_{\Omega} Q_1[\underline{U}] \\
                    \end{array}
                  \right)
\end{equation}
that
\begin{equation*}
\begin{split}
(SA[\underline{U}]\Lambda^s\partial_xU, \Lambda^sU)&=(\varepsilon\underline{u}\Lambda^s\eta_x, \Lambda^s\eta)
+(\underline{h}\Lambda^su_x, \Lambda^s\eta)+(\underline{h}\Lambda^s\eta_x, \Lambda^su)\\
&\quad+\bigg((\underline{\mathfrak{T}_1}(\varepsilon\underline{u}\cdot)+\underline{\kappa}_{\Omega} Q_1[\underline{U}])\Lambda^su_x, \Lambda^su\bigg)\\
&=:A_1+A_2+A_3+A_4.
\end{split}
\end{equation*}
We now focus on  the estimates of  bound for  each term $A_j$ ($j=1, 2, 3, 4$) step by step.

First, by using integrating by parts, there appears the relation
\begin{equation*}
A_1=(\varepsilon\underline{u}\Lambda^s\eta, \Lambda^s\eta_x)=-\frac{1}{2}(\varepsilon\underline{u_x}\Lambda^s\eta, \Lambda^s\eta),
\end{equation*}
which along with the Cauchy-Schwarz inequality yields
\begin{equation*}
|A_1|\leq \varepsilon \,\|\underline{u}_x\|_{L^\infty}\,E^s(U)^2.
\end{equation*}
Fro $A_2+A_3$, we notice the fact that
\begin{equation*}
|A_2+A_3|= |(\underline{h}_x\Lambda ^su, \Lambda^s\eta)|\leq \|\underline{h}_x\|_{L^\infty} \, E^s(U)^2,
\end{equation*}
which implies
\begin{equation*}
|A_2+A_3|\leq \varepsilon \,\|\underline{\eta}_x\|_{L^\infty}\,E^s(U)^2.
\end{equation*}
While for $A_4$, we split it into two parts
\begin{equation*}
A_4=\varepsilon(\underline{\kappa}_{\Omega}\underline{\mathfrak{T}}(\underline{u}\Lambda^su_x), \Lambda^su)+(\underline{\kappa}_{\Omega}Q_1[\underline{U}]\Lambda^su_x, \Lambda^su).
   =:A_{41}+A_{42}.
\end{equation*}
In view of  the definition of $\mathfrak{T} $, we have
\begin{equation*}
\begin{split}
A_{41} & = \varepsilon\bigg(\underline{\kappa}_{\Omega}\underline{h}\ \underline{u}\Lambda^su_x, \Lambda^su\bigg)+\frac{\varepsilon\mu}{3}
\bigg(\underline{h}^3(\underline{u}\Lambda^su_x)_x, \underline{\kappa}_{\Omega}\Lambda^su_x\bigg)\\
& \quad +\frac{\varepsilon\mu}{3}
\bigg(\underline{h}^3(\underline{u}\Lambda^su_x)_x, [\underline{\kappa}_{\Omega}]_x \Lambda^su\bigg)=:A_{411}+A_{412}+A_{413}.
\end{split}
\end{equation*}
Since
\begin{equation*}
\begin{split}
A_{411} & =\varepsilon\bigg(\underline{\kappa}_{\Omega}\underline{h} \underline{u}\Lambda^su_x, \Lambda^su\bigg)=-\frac{1}{2}\varepsilon\bigg((\underline{\kappa}_{\Omega}\underline{h}\underline{u})_x  \Lambda^su, \Lambda^su\bigg)\\
&=-\frac{1}{2}\varepsilon\bigg((\underline{\kappa}_{\Omega}(\underline{h}_x\underline{u}+\underline{h}\underline{u}_x)+(1-2\varepsilon\Omega \underline{u})^{-2} 2\varepsilon\Omega \,\underline{h}\underline{u} \underline{u}_x ) \Lambda^su, \Lambda^su\bigg),
\end{split}
\end{equation*}
it follows that
\begin{equation*}
\begin{split}
&|A_{411}| \leq C \, \varepsilon\, g_1(\|\underline{\eta}\|_{W^{1, \infty}}, \|\underline{u}\|_{W^{1, \infty}}) \|u\|_{H^s}^2
\end{split}
\end{equation*}
with
\begin{equation*}
\begin{split}
&g_1(\|\underline{\eta}\|_{W^{1, \infty}}, \|\underline{u}\|_{W^{1, \infty}}) =\varepsilon\|\underline{\eta}_x\|_{L^{\infty}}\|\underline{u}\|_{L^{\infty}}
+(1+\varepsilon\|\underline{\eta}\|_{L^{\infty}})(1+\varepsilon\Omega \|\underline{u}\|_{L^{\infty}})\|\underline{u}_x\|_{L^{\infty}}.
\end{split}
\end{equation*}
On the other hand, a direct computation reveals that
\begin{equation*}
\begin{split}
&\bigg(\underline{h}^3(\underline{u}\Lambda^su_x)_x,\  \underline{\kappa}_{\Omega} \Lambda^su_x\bigg)=\bigg(\underline{h}^3 \underline{u}(\Lambda^su_x)_x+\underline{h}^3 \underline{u}_x \Lambda^su_x ,\  \underline{\kappa}_{\Omega} \Lambda^su_x\bigg)\\
&=-\frac{1}{2}\bigg((\underline{\kappa}_{\Omega}\underline{h}^3 \underline{u})_x  \Lambda^su_x, \   \Lambda^su_x\bigg)
+\bigg(\underline{\kappa}_{\Omega} \underline{h}^3\underline{u}_x\Lambda^su_x, \ \Lambda^su_x)\bigg).
\end{split}
\end{equation*}
It is then deduced  that
\begin{equation*}
\begin{split}
|A_{412}| & \leq C\,\varepsilon \mu \bigg(\|(\underline{\kappa}_{\Omega}\underline{h}^3 \underline{u})_x\|_{L^{\infty}}+\|\underline{h}^3 \underline{u}_x\|_{L^{\infty}}\bigg)\|\Lambda^su_x\|_{L^2}^2\\
&\leq C\,\varepsilon \mu \bigg((1+\varepsilon\Omega \|\underline{u}\|_{L^{\infty}})\| \underline{u}_x\|_{L^{\infty}}\|\underline{h}\|_{L^{\infty}}^3 +\|\underline{h}\|_{L^{\infty}}^2\|\underline{h}_x\|_{L^{\infty}}  \|\underline{u}\|_{L^{\infty}}\bigg)\|\Lambda^su_x\|_{L^2}^2\\
&\leq C\,\varepsilon \, g_2(\|\underline{\eta}\|_{W^{1, \infty}}, \|\underline{u}\|_{W^{1, \infty}})\, \mu \, \|u_x\|_{H^s}^2
\end{split}
\end{equation*}
with
\begin{equation*}
\begin{split}
g_2 & (\|\underline{\eta}\|_{W^{1, \infty}}, \|\underline{u}\|_{W^{1, \infty}}) \\
&=(1+\varepsilon\Omega \|\underline{u}\|_{L^{\infty}})\| \underline{u}_x\|_{L^{\infty}}(1+\varepsilon\|\underline{\eta}\|_{L^{\infty}})^3 +(1+\varepsilon\|\underline{\eta} \|_{L^{\infty}})^2\varepsilon \|\underline{\eta}_x\|_{L^{\infty}}  \|\underline{u}\|_{L^{\infty}}.
\end{split}
\end{equation*}
On the other hand, a direct estimate of $A_{413}$ yields
\begin{equation*}
\begin{split}
|A_{413}| \leq &\ C\,\varepsilon\mu \|\underline{u}\|_{L^{\infty}} \bigg(\|\Lambda^su_x\|_{L^2}\|\Lambda^su\|_{L^2}(\|1+\varepsilon\eta\|_{L^{\infty}}^3\varepsilon\Omega \|\underline{u}_{xx}\|_{L^{\infty}}\\
&+\|1+\varepsilon\underline{\eta}\|_{L^{\infty}}^2\varepsilon\|\underline{\eta}_x\|_{L^{\infty}}\varepsilon\Omega \|\underline{u}_{x}\|_{L^{\infty}} +\|1+\varepsilon\underline{\eta}\|_{L^{\infty}}^3 \varepsilon^2\Omega^2 \|\underline{u}_{x}\|_{L^{\infty}}^2\big)\\
&+\|1+\varepsilon\underline{\eta}\|_{L^{\infty}}^3  \varepsilon\Omega \|\underline{u}_x\|_{L^{\infty}} \|\Lambda^su_x\|_{L^2}^2\bigg) \\
\leq & \  C\,\varepsilon^2\,  g_{3, 1}(\|\underline{\eta}\|_{W^{1, \infty}}, \|\underline{u}\|_{W^{1, \infty}}) \|u\|_{H^s}^2+C\,\varepsilon^2\, g_{3, 2}(\|\underline{\eta}\|_{W^{1, \infty}}, \|\underline{u}\|_{W^{1, \infty}}) \, \mu \, \|u_x\|_{H^s}^2\\
& + C\,\varepsilon^2\, g_{3, 3}(\|\underline{\eta}\|_{W^{1, \infty}}, \|\underline{u}\|_{W^{1, \infty}}) \, \sqrt{\mu}\, \|\underline{u}_{xx}\|_{L^{\infty}}\,  \sqrt{\mu} \, \|u_x\|_{H^s}\,  \|u\|_{H^s}.
\end{split}
\end{equation*}
with
\begin{equation*}
\begin{split}
g_{3, 1}& (\|\underline{\eta}\|_{W^{1, \infty}}, \|\underline{u}\|_{W^{1, \infty}})\\
&= \|\underline{u}\|_{L^{\infty}}^2 \mu \bigg(\|1+\varepsilon\underline{\eta}\|_{L^{\infty}}^2\varepsilon\|\underline{\eta}_x\|_{L^{\infty}}\varepsilon\Omega \|\underline{u}_{x}\|_{L^{\infty}} + \|1+\varepsilon\underline{\eta}\|_{L^{\infty}}^3 \varepsilon^2\Omega^2 \|\underline{u}_{x}\|_{L^{\infty}}^2\bigg)^2,\\
%\end{split}
%\end{equation*}
%\begin{equation*}
%\begin{split}
g_{3, 2}&(\|\underline{\eta}\|_{W^{1, \infty}}, \|\underline{u}\|_{W^{1, \infty}})\\
&= g_{3, 1}(\|\underline{\eta}\|_{W^{1, \infty}}, \|\underline{u}\|_{W^{1, \infty}})+ \|\underline{u}\|_{L^{\infty}}\|1+\varepsilon\underline{\eta}\|_{L^{\infty}}^3 \Omega \|\underline{u}_x\|_{L^{\infty}},
\end{split}
\end{equation*}
and
\begin{equation*}
\begin{split}
&g_{3, 3}(\|\underline{\eta}\|_{W^{1, \infty}}, \|\underline{u}\|_{W^{1, \infty}})=  \Omega \, \|\underline{u}\|_{L^{\infty}} \, \|1+\varepsilon\eta\|_{L^{\infty}}^3.
\end{split}
\end{equation*}
It is then adduced that
\begin{equation*}
|A_{41}|\leq \varepsilon \,C\, g_4(\|\underline{u}\|_{W^{1, \infty}}, \|\underline{\eta}\|_{W^{1, \infty}})(1+ E^s(\underline{U}))E^s(U)^2.
\end{equation*}
For $A_{42}$, it is inferred from the expression
\begin{equation*}
\begin{split}
|A_{42}| & =|(\underline{\kappa}_{\Omega}Q_1[\underline{U}]\Lambda^su_x, \Lambda^su)|=|\frac{2}{3}\varepsilon\mu\, \underline{\kappa}_{\Omega}(\underline{h}^3\underline{u}_x\Lambda^su_x)_x, \Lambda^su)|\\
& =\frac{2}{3}\varepsilon\mu|\underline{h}^3\underline{u}_x\Lambda^su_x, \underline{\kappa}_{\Omega} \Lambda^su_x+[\underline{\kappa}_{\Omega}]_x \Lambda^su)|
\end{split}
\end{equation*}
that
\begin{equation*}
|A_{42}|\leq \varepsilon g_5(|\underline{u}|_{w^{1, \infty}}, |\underline{\eta}|_{w^{1, \infty}})E^s(U)^2.
\end{equation*}
This thus implies that
\begin{equation*}
|A_{4}|\leq \varepsilon C \, g_6(\|\underline{u}\|_{W^{1, \infty}}, \|\underline{\eta}\|_{W^{1, \infty}})(1+ E^s(\underline{U}))E^s(U)^2.
\end{equation*}

We now turn to estimate of $([\Lambda^s, A[\underline{U}]]\partial_xU, S\Lambda^sU)$. We first utilize definitions of  $A[\underline{U}]$ and $S$ to get
\begin{equation*}
\begin{split}
&([\Lambda^s, A[\underline{U}]]\partial_xU, S\Lambda^sU)=([\Lambda^s, \varepsilon \underline{u}]\eta_x, \Lambda^s\eta)
+([\Lambda^s, \underline{h}]u_x, \Lambda^s\eta)\\
&\quad+([\Lambda^s, \underline{\mathfrak{T}}^{-1}(\underline{\kappa}_{\Omega}^{-1}\underline{h}\cdot)]\eta_x, \underline{\kappa}_{\Omega}\underline{\mathfrak{T}}\Lambda^su)+([\Lambda^s, \varepsilon\underline{u}]u_x, \underline{\kappa}_{\Omega}\underline{\mathfrak{T}}\Lambda^su)
+([\Lambda^s, \underline{\mathfrak{T}}^{-1}Q_1[\underline{U}]]u_x, \underline{\kappa}_{\Omega}\underline{\mathfrak{T}}\Lambda^su)\\
& \quad =:B_1+B_2+B_3+B_4+B_5.
\end{split}
\end{equation*}
Since $B_1+B_2=([\Lambda^s, \varepsilon \underline{u}]\eta_x, \Lambda^s\eta)
+([\Lambda^s, \varepsilon\underline{\eta}]u_x, \Lambda^s\eta)$ and $s>\frac{3}{2}$, we apply commutator estimates Lemma \ref{lemma-comm} to get
\begin{equation*}
|B_1+B_2|\leq\varepsilon C(E^s(\underline{U}))E^s(U)^2.
\end{equation*}
For $B_4$, we use the explicit expression of
$\underline{\mathfrak{T}}$ to obtain
\begin{equation*}
\begin{split}
B_4=&([\Lambda^s, \varepsilon \underline{u}]u_x, \underline{\kappa}_{\Omega}\underline{h}\Lambda^s u)+\frac{\mu}{3}(\partial_x[\Lambda^s, \varepsilon \underline{u}]u_x, \underline{\kappa}_{\Omega} \underline{h}^3\Lambda^s u_x)+\frac{\mu}{3}([\Lambda^s, \varepsilon \underline{u}]u_x, [\underline{\kappa}_{\Omega}]_x\underline{h}^3\Lambda^s u_x),
\end{split}
\end{equation*}
which, together with the fact that
\begin{equation*}
\partial_x[\Lambda^s, f]g=[\Lambda^s, f_x]g+[\Lambda^s, f]g_x
\end{equation*}
and then the Cauchy-Schwartz inequality and commutator estimates in Lemma \ref{lemma-comm}, yields
\begin{equation*}
|B_4|\leq\varepsilon C(E^s(\underline{U}))E^s(U)^2.
\end{equation*}

Next step is to deal  with $B_3=([\Lambda^s, \underline{\mathfrak{T}}^{-1}(\underline{\kappa}_{\Omega}^{-1}\underline{h}\cdot)]\eta_x, \underline{\kappa}_{\Omega}\underline{\mathfrak{T}}\Lambda^su)$. In view of  definition of the operator $\underline{\mathfrak{T}}_1$, we get for any $f$ and $g$
\begin{equation}\label{dual-ope-1}
(f, \, \underline{\mathfrak{T}}_1\,g)=(\underline{\mathfrak{T}} \,f, \,\underline{\kappa}_{\Omega} \,g)-\frac{\mu}{3}\bigg(2[\underline{\kappa}_{\Omega}]_x\underline{h}^3f_x+(3[\underline{\kappa}_{\Omega}]_x\underline{h}^2\underline{h}_x
+[\underline{\kappa}_{\Omega}]_{xx}h^3)f, g\bigg),
\end{equation}
which implies
\begin{equation}\label{dual-ope-2}
\begin{split}
&B_3=(\underline{\mathfrak{T}} \,[\Lambda^s, \underline{\mathfrak{T}}^{-1}(\underline{\kappa}_{\Omega}^{-1}\underline{h}\cdot)]\eta_x, \,\underline{\kappa}_{\Omega} \,\Lambda^su)\\
&\qquad-\frac{\mu}{3}\bigg(2[\underline{\kappa}_{\Omega}]_x\underline{h}^3f_x+(3[\underline{\kappa}_{\Omega}]_x\underline{h}^2\underline{h}_x
+[\underline{\kappa}_{\Omega}]_{xx}h^3)f, \Lambda^su\bigg)=:B_{31}+B_{32}
\end{split}
\end{equation}
with $f=[\Lambda^s, \underline{\mathfrak{T}}^{-1}(\underline{\kappa}_{\Omega}^{-1}\underline{h}\cdot)]\eta_x$.
It is noted that
\begin{equation*}
\begin{split}
\underline{\mathfrak{T}}[\Lambda^s, \underline{\mathfrak{T}}^{-1}\underline{\kappa}_{\Omega}^{-1}\underline{h}]\eta_x & =\underline{\mathfrak{T}}[\Lambda^s, \underline{\mathfrak{T}}^{-1}]\underline{\kappa}_{\Omega}^{-1}\underline{h}\eta_x+[\Lambda^s, \underline{\kappa}_{\Omega}^{-1}\underline{h}]\eta_x\\
&=-[\Lambda^s, \underline{\mathfrak{T}}]\underline{\mathfrak{T}}^{-1}\underline{\kappa}_{\Omega}^{-1}\underline{h}\eta_x
+[\Lambda^s, \underline{\kappa}_{\Omega}^{-1}\underline{h}]\eta_x.
\end{split}
\end{equation*}
It then follows from the definition  of $\underline{\mathfrak{T}}$ that
\begin{equation*}
\begin{split}
\underline{\mathfrak{T}}[\Lambda^s, \underline{\mathfrak{T}}^{-1}\underline{\kappa}_{\Omega}^{-1}\underline{h}]\eta_x=&-[\Lambda^s, \underline{h}]\underline{\mathfrak{T}}^{-1}\underline{\kappa}_{\Omega}^{-1}\underline{h}\eta_x
+\frac{\mu}{3}\partial_x\{[\Lambda^s, \underline{h}^3]\partial_x(\underline{\mathfrak{T}}^{-1}\underline{\kappa}_{\Omega}^{-1}\underline{h}\eta_x)\}+[\Lambda^s, \underline{\kappa}_{\Omega}^{-1}\underline{h}]\eta_x,
\end{split}
\end{equation*}
which gives rise to
\begin{equation*}
\begin{split}
|B_{31}| & =\bigg|(-[\Lambda^s, \underline{h}]\underline{\mathfrak{T}}^{-1}\underline{\kappa}_{\Omega}^{-1}\underline{h}\eta_x
+\frac{\mu}{3}\partial_x\{[\Lambda^s, \underline{h}^3]\partial_x(\underline{\mathfrak{T}}^{-1}\underline{\kappa}_{\Omega}^{-1}\underline{h}\eta_x)\}+[\Lambda^s, \underline{\kappa}_{\Omega}^{-1}\underline{h}]\eta_x, \underline{\kappa}_{\Omega} \,\Lambda^su)\bigg|\\
&\leq |([\Lambda^s, \underline{h}]\underline{\mathfrak{T}}^{-1}\underline{\kappa}_{\Omega}^{-1}\underline{h}\eta_x, \underline{\kappa}_{\Omega} \,\Lambda^su)|
+\frac{\mu}{3}|([\Lambda^s, \underline{h}^3]\partial_x(\underline{\mathfrak{T}}^{-1}\underline{\kappa}_{\Omega}^{-1}\underline{h}\eta_x), [\underline{\kappa}_{\Omega}]_{x} \,\Lambda^su+\underline{\kappa}_{\Omega} \,\Lambda^su_x)|\\
&\quad +\varepsilon |([\Lambda^s, 2\Omega\underline{u}-\underline{\eta}+2\varepsilon \Omega \underline{u}\underline{\eta}]\eta_x, \underline{\kappa}_{\Omega} \,\Lambda^su)|.
\end{split}
\end{equation*}
Hence, it is thereby inferred  that
\begin{equation*}
\begin{split}
|B_{31}| & \leq \|[\Lambda^s, \varepsilon\underline{\eta}]\underline{\mathfrak{T}}^{-1}\underline{\kappa}_{\Omega}^{-1}\underline{h}\eta_x\|_{L^2}\|\underline{\kappa}_{\Omega} \,\Lambda^su\|_{L^2}\\
& \quad +\frac{\mu}{3}\|[\Lambda^s, \underline{h}^3]\partial_x(\underline{\mathfrak{T}}^{-1}\underline{\kappa}_{\Omega}^{-1}\underline{h}\eta_x)\|_{L^2} (\|[\underline{\kappa}_{\Omega}]_{x} \,\Lambda^su\|_{L^2}+\|\underline{\kappa}_{\Omega} \,\Lambda^su_x\|_{L^2})\\
&\quad +\varepsilon\|[\Lambda^s, 2\Omega\underline{u}-\underline{\eta}+2\varepsilon \Omega \underline{u}\underline{\eta}]\eta_x\|_{L^2}\|\underline{\kappa}_{\Omega} \,\Lambda^su\|_{L^2}=:B_{311}+B_{312}+B_{313}.
\end{split}
\end{equation*}
It is then deduced from   Lemmas \ref{lemma-comm} and \ref{lem-operator-2} that
\begin{equation*}
\begin{split}
B_{311} & \leq C \varepsilon\|\underline{\eta}\|_{H^{s}}\|\underline{\mathfrak{T}}^{-1}\underline{\kappa}_{\Omega}^{-1}\underline{h}\eta_x\|_{H^{s-1}}\|u\|_{H^{s}}\\
&\leq C\left(\frac{1}{b_0}, \|\underline{h}-1\|_{H^s} \right)(1+\varepsilon \Omega\|\underline{u}\|_{H^{s-1}})(1+\varepsilon\|\underline{\eta}\|_{H^{s-1}})\varepsilon\|\underline{\eta}\|_{H^{s}}\|\eta_x\|_{H^{s-1}}\|u\|_{H^{s}},\\
B_{312} & \leq C \mu \|\varepsilon \eta_x \underline{h}^2\|_{H^{s-1}}\|\underline{\mathfrak{T}}^{-1}\underline{\kappa}_{\Omega}^{-1}\underline{h}\eta_x\|_{H^{s}} (\|\varepsilon \Omega\|\underline{u}_x\|_{L^{\infty}}\|u\|_{H^s}+(1+\|\varepsilon \Omega\|\underline{u}\|_{L^{\infty}})\|u_x\|_{H^s})\\
&\leq C(\frac{1}{b_0}, \|\underline{h}-1\|_{H^s})(1+\varepsilon \Omega\|\underline{u}\|_{H^{s}})(1+\varepsilon\|\underline{\eta}\|_{H^{s}})^3 \mu \varepsilon  \| \eta_x\|_{H^{s-1}}\|\underline{h}\eta_x\|_{H^{s}} \\
&\quad\times (\|\varepsilon \Omega\|\underline{u}_x\|_{L^{\infty}}\|u\|_{H^s}(1+\|\varepsilon \Omega\|\underline{u}\|_{L^{\infty}})\|u_x\|_{H^s}),
\end{split}
\end{equation*}
and
\begin{equation*}
\begin{split}
B_{313} &\leq C \varepsilon\|\Omega\underline{u}-\underline{\eta}+2\varepsilon \Omega \underline{u}\underline{\eta}\|_{H^s}\|\eta_x\|_{H^{s-1}}  \|u\|_{H^s}\\
&\leq C \varepsilon(\|\underline{u}\|_{H^s}+\|\underline{\eta}\|_{H^s}+\varepsilon \|\underline{u}\|_{H^s}\|\underline{\eta}\|_{H^s})\|\eta_x\|_{H^{s-1}}  \|u\|_{H^s} .
\end{split}
\end{equation*}
While for $B_{32}$, we have
\begin{equation*}
\begin{split}
|B_{32}| &\leq \frac{\mu}{3}\bigg(|(f, (2[\underline{\kappa}_{\Omega}]_x\underline{h}^3)_x \Lambda^su+2[\underline{\kappa}_{\Omega}]_x\underline{h}^3\Lambda^su_x)|+(3[\underline{\kappa}_{\Omega}]_x\underline{h}^2\underline{h}_x
+[\underline{\kappa}_{\Omega}]_{xx}h^3)f, \Lambda^su)|\bigg)\\
&\leq C \mu\varepsilon \|f\|_{L^2}\bigg(\|u\|_{H^s}(\|\underline{u}\|_{H^s}(1+\varepsilon\|\underline{\eta}\|_{H^s})^3+\|\underline{u}_x\|_{L^{\infty}}
(1+\varepsilon\|\underline{\eta}\|_{L^{\infty}})^2 \varepsilon\|\underline{\eta}_x\|_{L^{\infty}}\\
&\qquad \qquad \qquad +\|\underline{u}_{xx}\|_{L^{\infty}}
(1+\varepsilon\|\underline{\eta}\|_{L^{\infty}})^3  )+\|\underline{u}_{x}\|_{L^{\infty}}
(1+\varepsilon\|\underline{\eta}\|_{L^{\infty}})^3 \|u_x\|_{H^s}\bigg),
\end{split}
\end{equation*}
which along with
\begin{equation*}
\begin{split}
  \|f\|_{L^2}&=\|[\Lambda^s, \underline{\mathfrak{T}}^{-1}(\underline{\kappa}_{\Omega}^{-1}\underline{h}\cdot)]\eta_x\|_{L^2}\leq C (\|\underline{\mathfrak{T}}^{-1}(\underline{\kappa}_{\Omega}^{-1}\underline{h}\Lambda^s\eta_x)\|_{L^2}
  +\|\underline{\mathfrak{T}}^{-1}(\underline{\kappa}_{\Omega}^{-1}\underline{h}\eta_x)\|_{H^s})\\
  &\leq C\left(\frac{1}{b_0}, \|\underline{h}-1\|_{H^s} \right)(1+\varepsilon \Omega\|\underline{u}\|_{H^{s}})(1+\varepsilon\|\underline{\eta}\|_{H^{s}})\|\underline{h}\eta_x)\|_{H^s}
  \end{split}
\end{equation*}
implies
\begin{equation*}
|B_{32}|\leq \varepsilon C(E^s(\underline{U}))E^s(U)^2.
\end{equation*}
Therefore, we deduce that
\begin{equation*}
|B_3|\leq \varepsilon C(E^s(\underline{U}))E^s(U)^2.
\end{equation*}
To control $B_5=([\Lambda^s, \underline{\mathfrak{T}}^{-1}Q_1[\underline{U}]]u_x, \underline{\kappa}_{\Omega}\underline{\mathfrak{T}}\Lambda^su)$, let us first write
\begin{equation*}
\underline{\mathfrak{T}}[\Lambda^s, \underline{\mathfrak{T}}^{-1}Q_1[\underline{U}]]u_x=-[\Lambda^s, \underline{\mathfrak{T}}]\underline{\mathfrak{T}}^{-1}Q_1[\underline{U}]u_x+[\Lambda^s, Q_1[\underline{U}]]u_x
\end{equation*}
so, that
\begin{equation*}
\begin{split}
&\underline{\mathfrak{T}}[\Lambda^s, \underline{\mathfrak{T}}^{-1}Q_1[\underline{U}]]u_x\\
&=-[\Lambda^s, \underline{h}]\underline{\mathfrak{T}}^{-1}
Q_1[\underline{U}]u_x+\frac{\mu}{3}\partial_x\{[\Lambda^s, \underline{h}^3]\partial_x(\underline{\mathfrak{T}}^{-1}Q_1[\underline{U}]u_x)\}+[\Lambda^s, Q_1[\underline{U}]]u_x.
\end{split}
\end{equation*}
By using the explicit expression of $Q_1[\underline{U}]$:
\begin{equation*}
Q_1[\underline{U}]f=\frac{2}{3}\varepsilon\mu\partial_x(\underline{h}^3\underline{u}_xf)-2\Omega \underline{h}^2f
\end{equation*}
and the fact that
\begin{equation*}
\begin{split}
(\partial_x\{[\Lambda^s, \underline{h}^3]\partial_x(\underline{\mathfrak{T}}^{-1}Q_1[\underline{U}]u_x)\}, \Lambda^su)
&=-\frac{2}{3}\varepsilon\mu([\Lambda^s, \underline{h}^3]\partial_x(\underline{\mathfrak{T}}^{-1}\partial_x(\underline{h}^3
\underline{u}_xu_x)), \Lambda^su_x)\\
&\quad +(\{[\Lambda^s, \underline{h}^3]\partial_x(\underline{\mathfrak{T}}^{-1}2\Omega \underline{h}^2u_x)\}, \Lambda^su_x),
\end{split}
\end{equation*}
and then repeating the similar argument in the estimate of $B_3$, it is found that
\begin{equation*}
|B_5|\leq \varepsilon C(E^s(\underline{U}))E^s(U)^2.
\end{equation*}
For $(\Lambda^su, [\partial_t, \underline{\kappa}_{\Omega}\underline{\mathfrak{T}}]\Lambda^su)$, we first rewrite it as follows
\begin{equation*}
\begin{split}
(\Lambda^su, [\partial_t, \underline{\kappa}_{\Omega}\underline{\mathfrak{T}}]\Lambda^su)=&\ (\Lambda^su, ( \underline{\kappa}_{\Omega} \underline{h})_t\Lambda^su)+\frac{\mu}{3}
(\Lambda^su_x, (\underline{\kappa}_{\Omega} \underline{h}^3)_t\Lambda^su_x)\\
&+\frac{\mu}{3}
(\Lambda^su, ([\underline{\kappa}_{\Omega}]_x \underline{h}^3)_t\Lambda^su_x).
\end{split}
\end{equation*}
By making use of this form, we get
\begin{equation*}
\begin{split}
&|(\Lambda^su, [\partial_t, \underline{\kappa}_{\Omega}\underline{\mathfrak{T}}]\Lambda^su)|\leq (\|u\|_{H^s}^2 +\frac{\mu}{3}
\|u_x\|_{H^s}^2)\|(\underline{\kappa}_{\Omega} \underline{h}^3)_t\|_{L^{\infty}}+\frac{\mu}{3}
\|u\|_{H^s}\|u_x\|_{H^s}\|([\underline{\kappa}_{\Omega}]_x \underline{h}^3)_t\|_{L^{\infty}}\\
& \quad \leq C\, \varepsilon\,(\|u\|_{H^s}^2 +\frac{\mu}{3}
\|u_x\|_{H^s}^2)(\|\underline{u}_t\|_{L^{\infty}}+\|\underline{u}_x\|_{L^{\infty}}\|\underline{u}_t\|_{L^{\infty}}
+\sqrt{\mu}\|\underline{u}_{xt}\|_{L^{\infty}}+\|\underline{\eta}_t\|_{L^{\infty}})\\
&\quad\times(1+\|\underline{\eta}_t\|_{L^{\infty}}+\varepsilon \|\underline{\eta}\|_{L^{\infty}})^3 \leq \varepsilon\,C(E^s(\underline{U}), \|\underline{\eta}_t\|_{L^\infty}, \|\underline{u}_t\|_{L^\infty}, \sqrt{\mu}\|\underline{u}_{tx}\|_{L^\infty})E^s(U)^2.
\end{split}
\end{equation*}
It is thereby adduced that
\begin{equation*}
\begin{split}
&e^{\varepsilon\lambda t} \frac{d}{dt}(e^{-\varepsilon\lambda t}E^s(U)^2)\leq \varepsilon\,\bigg(C(E^s(\underline{U}), \|\underline{\eta}_t\|_{L^\infty}, \|\underline{u}_t\|_{L^\infty}, \sqrt{\mu}\|\underline{u}_{tx}\|_{L^\infty})- \lambda\bigg) E^s(U)^2+C\, \varepsilon \,\|F\|_{X^s}^2.
\end{split}
\end{equation*}
We now take  $\lambda=\lambda_T$ large enough (depending on $\sup_{t\in[0, \frac{T}{\varepsilon}]}C(E^s(\underline{U}), \|\underline{\eta}_t\|_{L^\infty}, \|\underline{u}_t\|_{L^\infty}, \sqrt{\mu}\|\underline{u}_{tx}\|_{L^\infty})$) so that  the first term of the right-hand side is negative for all $t\in[0, \frac{T}{\varepsilon}]$.
This then follows that
\begin{equation*}
\forall t\in \left[0, \frac{T}{\varepsilon}\right],\quad e^{\varepsilon\lambda t}\frac{d}{dt}(e^{-\varepsilon\lambda t}E^s(U)^2)
\leq C\, \varepsilon \,\|F\|_{X^s}^2.
\end{equation*}
Integrating this differential inequality yields that
\begin{equation*}
\forall t\in \left[0, \frac{T}{\varepsilon} \right],\quad E^s(U(t))\leq e^{\varepsilon\lambda_Tt}E^s(U_0)+C\, \varepsilon \,\int_0^{t}e^{\varepsilon\lambda_T(t-\tau)}\|F(\tau)\|_{X^s}^2\, d\tau.
\end{equation*}
This completes  the proof of Lemma \ref{thm-linearized-1}.
\end{proof}

\subsection{Proof of Theorem \ref{thm-local}}
Firstly, uniqueness and continuity with respect to the initial data are an immediate
consequence of the following result.
\begin{lemma}\label{lem-uniqueness-1} Let $ s $ and $b_0$ be as in the statement of Theorem \ref{thm-local}, $U^{(1)} $ and $U^{(2)} $ be two given solutions of the initial-value problem \eqref{R-GN-3} with the initial data
$U^{(1)}_0,\, U^{(2)}_0 \in X^s$ satisfying $U^{(1)}, \, U^{(2)} \in X^s_T$. Then for every
$t \in [0, \frac{T}{\varepsilon}]:$
\begin{equation}\label{unique-1}
\begin{split}
&\|U^{(1)}(t)-U^{(2)}(t)\|_{X^{s-1}}\leq \|U^{(1)}_0-U^{(2)}_0\|_{X^{s-1}} e^{C\varepsilon\int_{0}^{t}(\|U^{(1)}(\tau)\|_{X^{s}}+\|U^{(2)}(\tau)\|_{X^{s}}+1)\,d\tau}.
\end{split}
\end{equation}
\end{lemma}
\begin{proof}
Denote $U^{(12)} = U^{(2)}-U^{(1)}$. Then
$U^{(12)} \in X_T^s$ solves the transport equations
\begin{equation}\label{unique-2}
    \begin{cases}
        &\partial_t{U^{(12)}} + A[U^{(1)}] \partial_x{U^{(12)}}+ (A[U^{(2)}]-A[U^{(1)}]) \partial_x{U^{(2)}}= 0; \\
        & U^{(12)}|_{t = 0} = U^{(2)}_0-U^{(1)}_0.
    \end{cases}
\end{equation}
For $E^{s-1}(U^{(12)})^2:=(\Lambda^s U^{(12)}, S(U^{(1)})\, \Lambda^s U^{(12)})$, we get
\begin{equation}\label{unique-3}
\begin{split}
\frac{d}{dt}(E^{s-1}(U^{(12)})^2) & =2(\Lambda^{s-1}\eta^{(12)}, \Lambda^{s-1}\eta^{(12)}_t)+(\Lambda^{s-1}u^{(12)}_t, \underline{\mathfrak{T}_1}\Lambda^{s-1}u^{(12)})\\
&\quad +(\Lambda^{s-1}u^{(12)}, \underline{\mathfrak{T}_1}\Lambda^{s-1}u^{(12)}_t)
+(\Lambda^{s-1}u^{(12)}, [\partial_t, \underline{\mathfrak{T}_1}]\Lambda^{s-1}u^{(12)}).
\end{split}
\end{equation}
On the other hand, applying Lemma \ref{lem-operator-diff-1} ensures that
\begin{equation}\label{unique-4}
\begin{split}
\|(A[U^{(2)}]-A[U^{(1)}]) \partial_x{U^{(2)}}\|_{X^{s-1}} \leq \varepsilon\, C\,\|U^{(2)}\|_{X^{{s}}} \|U^{(12)}\|_{X^{s-1}}.
\end{split}
\end{equation}
Consequently, in view of  Lemmas \ref{lem-operator-diff-1} and \ref{thm-linearized-1}, and \eqref{unique-2}-\eqref{unique-3},   the advertised result can be obtained  by repeating the argument in the proof of Lemma \ref{thm-linearized-1}.
\end{proof}

Next, we shall use the classical Friedrichs' regularization method to construct the approximate solutions to the Green-Naghdi equations \eqref{R-GN-3}.

\begin{lemma}\label{lem-linear-appro-1}
Let $U_0$, $ s $, and $b_0$ be as in the statement of Theorem \ref{thm-local}. Assume that $U^{(0)}:= U_0$.  Then there
exist a sequence of times $(T^{(n)})_{n \in
\mathbb{N}}$ and smooth functions $(U^{(n)})_{n \in
\mathbb{N}} \in \mathcal{C}([0, \frac{T^{(n)}}{\varepsilon}]; X^{s})$
solving the following linear transport equation by induction:
\begin{equation}\label{appro-eqns-1}(GN_n)~~~~~~
\begin{cases}
 & \partial_{t}U^{(n+1)} + A[U^{(n)}]\partial_{x}U^{(n+1)}= 0, \quad t > 0,\, x\in \mathbb{R},\\
    & U^{(n+1)}|_{t = 0} = U_{0}, \quad x\in \mathbb{R}.
\end{cases}
\end{equation}
Moreover, there is a positive time $T$ ($<T^{(n)}$ for all $n \in \mathbb{N}$) such that the corresponding solutions satisfy
the following properties:

(i). $(U^{(n)})_{n \in \mathbb{N}}$ is uniformly bounded in $X^{s}_T$.

(ii). $(U^{(n)})_{n \in \mathbb{N}}$ is a Cauchy sequence in
$X^{s-1}_T$.
\end{lemma}

\begin{proof}%[Proof of Theorem \ref{thm-local}]
  By Lemma \ref{thm-linearized-1}, we know that, for every $n\in \mathbb{N}$, there is a positive time $T^{(n)}$ and a unique solution $ U^{(n + 1)} \in \mathcal{C}([0, \frac{T^{(n)}}{\varepsilon}]; X^{s})$ to \eqref{appro-eqns-1}. Moreover, we may verify that $ U^{(n + 1)} $ satisfies the inequality
\begin{equation*}
    E^{s}(U^{(n + 1)}(t)) \leq
        e^{\varepsilon \lambda_{T}^{(n)}t}E^{s}(U_{0})
\end{equation*}
if
\begin{equation*}
    \lambda_{T}^{(n)} \geq  \sup_{t\in[0, \frac{1}{\varepsilon}T^{(n)} ]}C(E^s(U^{(n)} ), \|\eta^{(n)}_t\|_{L^\infty}, \|u^{(n)}_t\|_{L^\infty}, \sqrt{\mu}\|u^{(n)}_{tx}\|_{L^\infty})
\end{equation*}
and the condition
\eqref{initial-condition-1} holds for $n\in \mathbb{N}$.

In fact, we suppose by the induction argument that
\begin{equation}\label{induction-1}
  \sup_{t\in[0, \frac{1}{\varepsilon}T^{(n)} ]}E^{s}(U^{(n)}(t)) \leq 2 \,E^{s}(U_{0}),
\end{equation}
which is already satisfied for the case $n=0$.
Hence, thanks to the equation \eqref{appro-eqns-1} and the Sobolev embedding $H^s(\mathbb{R}) \hookrightarrow L^{\infty}(\mathbb{R}) $  for $s>\frac{1}{2}$, we get
\begin{equation}\label{uniform-bdd-1}
 \|\eta^{(n)}_t\|_{L^\infty}^2+\|u^{(n)}_t\|_{L^\infty}^2+\sqrt{\mu}\|u^{(n)}_{tx}\|_{L^\infty}^2\leq C_0 E^{s}(U^{(n)}(t)) \leq 2 C_0\,E^{s}(U_{0}),
\end{equation}
and then
\begin{equation*}
 \sup_{t\in[0, \frac{1}{\varepsilon}T^{(n)} ]}C \left(E^s(U^{(n)} ), \|\eta^{(n)}_t\|_{L^\infty}, \|u^{(n)}_t\|_{L^\infty}, \sqrt{\mu}\|u^{(n)}_{tx}\|_{L^\infty} \right) \leq C_1(E^{s}(U_{0})).
\end{equation*}
It follows that if we take $\lambda_{T}^{(n)}=C_1(E^{s}(U_{0}))$ and  $ T^{(n)} =\frac{1}{2\varepsilon} C_1(E^{s}(U_{0}))=:T_1$, then we get
\begin{equation}\label{uniform-bdd-2}
  \sup_{t \in \big[{0,\ \frac{T_1}{\varepsilon}}\big]} {E^{s} \left(U^{(n + 1)}(t) \right)} \leq e^{\frac{1}{2}}\,E^{s}(U_{0}) \leq 2\,E^{s}(U_{0}) \quad (\forall \, n \in \mathbb{N}).
\end{equation}
We  now verify that the condition \eqref{initial-condition-1} holds for every $n \in \mathbb{N}$ if we take the positive time $T$ small enough (independent of $\varepsilon$ and $\mu$). In fact, since
\begin{equation*}
   h^{(n)}= h_0+\varepsilon +
      \int_{0}^{t}{\partial_{t}\eta^{(n)}}(\tau)\, d\tau  \quad \mbox{and} \quad  u^{(n)} = u^{(n)}|_{t = 0} +
      \int_{0}^{t}{\partial_{t}u^{(n)}}(\tau)\, d\tau,
\end{equation*}
we get from \eqref{uniform-bdd-1} that
\begin{equation}\label{uniform-bdd-3}
\begin{split}
&\inf_{x\in \mathbb{R}} h^{(n)}(t) \geq \inf_{x\in \mathbb{R}} h_0 -2 \, t\, \varepsilon\,(E^{s}(U_{0}))^{\frac{1}{2}} ,\\
& \inf_{x\in \mathbb{R}} (1-2\Omega \varepsilon u^{(n)}(t)) \} \geq \inf_{x\in \mathbb{R}} (1-2\Omega \varepsilon u_0) \} -4\Omega \varepsilon \, t
    \,(E^{s}(U_{0}))^{\frac{1}{2}}.
\end{split}
\end{equation}
Therefore, taking $T=\min\{T_1, ((4+8\Omega)b_0)^{-1}(E^{s}(U_{0}))^{-\frac{1}{2}}\}$, we obtain that, for every $n \in \mathbb{N}$, $t \in [0, \frac{1}{\varepsilon}T]$,  the condition \eqref{initial-condition-1} holds with $ b_{0} $ replace by $ b_{0}/2 $. This  completes the proof of Lemma \ref{lem-linear-appro-1} by a classical bootstrap argument.
\end{proof}

\begin{proof}[Proof of Theorem \ref{thm-local}] Thanks to Lemmas \ref{thm-linearized-1}, \ref{lem-uniqueness-1}, and \ref{lem-linear-appro-1}, we may readily find, by the standard argument in hyperbolic PDEs, a positive maximal existence time $T_{max}>0$, uniformly bounded from below with respect to $\varepsilon, \, \mu \in (0, 1)$, such that the Green-Naghdi equations \eqref{R-GN-3} admit a unique solution $U = (\eta , u)^T \in X^s_{T_{max}}$ preserving the condition \eqref{initial-condition-1} for any $ t \in [0, \frac{T_{max}}{\varepsilon})$. In particular if $T_{max} < +\infty$, there hold \eqref{blowup-cond-1} and \eqref{blowup-cond-2}.

Finally, a direction computation (multiplying the first equation in \eqref{R-GN-1} by $\eta$ and the second equation by $u$ then summing up the two) leads to the conservation of energy $E(\eta, u) = \|\eta\|_{L^2}^2+(\mathfrak{T}u, u)=\int_{\mathbb{R}}(\eta^2+h\, u^2+\frac{\mu}{3}h^3u_x^2)\,dx$. This ends the proof of Theorem \ref{thm-local}.
\end{proof}

%%%%%%%%%%%%%%%%%%%%%%%%%%%%%%%%%%%%%%%%%%%%%%%%%%%%%%%%%%%%%%
%%%%%%%%%%%%%%%%%%%%%%%%%%%%%%%%%%%%%%%%%%%%%%%%%%%%%%%%%%%%%
\renewcommand{\theequation}{\thesection.\arabic{equation}}
\setcounter{equation}{0}

%%%%%%%%%%%%%%%%%%%%%%%%%%%%%%%%%%%%%%%%%%%%%%%%%%%%%%%%%%%%%%
%%%%%%%%%%%%%%%%%%%%%%%%%%%%%%%%%%%%%%%%%%%%%%%%%%%%%%%%%%%%%%%%%%

\section{Rigorous justification of the unidirectional approximations}\label{Sec_just}

\begin{theorem}\label{thm-justification-1}
 Given $\mu_0 > 0$ and $M > 0$. Let $p \in \mathbb{R}$, $\theta \in  [0, 1]$, and $\alpha$, $\beta$, $\gamma$ and $\delta$ be as in Proposition \ref{prop-two-parameters}. If $\beta<0$ then there exists $D > 0$ and $T > 0$ such that for all $u_0\in H^{s+D+1}(\mathbb{R})$, there hold:

(1) there is a unique family $(u^{\varepsilon, \mu}, \eta^{\varepsilon, \mu})_{(\varepsilon, \mu)\in \mathcal{P}_{\mu_0, M}} \in C([0, \frac{T}{\varepsilon}]; H^{s+D}(\mathbb{R})^2)$ given by the resolution of \eqref{GBBM-1} with initial condition $u_0$;

(2) there is a unique family $(\underline{u}^{\varepsilon, \mu}, \underline{\eta}^{\varepsilon, \mu})_{(\varepsilon, \mu)\in \mathcal{P}_{\mu_0, M}} \in C([0, \frac{T}{\varepsilon}]; H^{s+D}(\mathbb{R})^2)$ solving the R-GN equations \eqref{R-GN-2} with initial condition $(u^{\varepsilon, \mu}, \eta^{\varepsilon, \mu})|_{t=0}$.
Moreover, for all $(\varepsilon, \mu) \in \mathcal{P}_{\mu_0, M}$, there holds that
\begin{equation*}
\forall \, t \in [0, \frac{T}{\varepsilon}], \quad \|\underline{u}^{\varepsilon, \mu}-u^{\varepsilon, \mu}\|_{L^{\infty}([0, t] \times \mathbb{R})}+\|\underline{\eta}^{\varepsilon, \mu}-\eta^{\varepsilon, \mu}\|_{L^{\infty}([0, t] \times \mathbb{R})} \leq C\, \mu^2\,t,
\end{equation*}
where the constant $C$ is independent of $\varepsilon$ and $\mu$.
\end{theorem}
\begin{proof}
Part (1) can be obtained directly from Theorem \ref{thm-r-CH-local}. In view of Theorem \ref{prop-two-parameters} and Remark \ref{rmk-CH-eqns-1}, we know that the family $(u^{\varepsilon, \mu}, \eta^{\varepsilon, \mu})_{(\varepsilon, \mu)\in \mathcal{P}_{\mu_0, M}} $ is consistent with the R-GN equations \eqref{R-GN-2}, so that the second part of the theorem and the error estimate follow from the well-posedness theorem (Theorem \ref{thm-local}) and stability of the R-GN equations (Lemma \ref{lem-uniqueness-1}).
\end{proof}
\vskip 0.2cm

\noindent {\bf Acknowledgments.}
The work of Chen is partially supported by the NSF grant DMS-1613375. The work of Gui is supported in part by the NSF-China under the grants 11571279, 11331005, and the Foundation FANEDD-201315.  The work of Liu is partially supported by the Simons Foundation grant 499875.

%%%%%%%%%%%%%%%%%%%%%%%%%%%%%%%%%%%%%%%%%%%%%%%%%%%%%%%%%%%%%%%%%%%%%%%%%%

%%%%%%%%%%%%%%%%%%%%%%%%%%%%%%%%%%%%%%%%%%%%%%%%%%%%%%%%%%%%%%%%%%%%%%%%%%%%%%%%%
%%%%%%%%%%%%%%%%%%%%%%%%%%%%%%%%%%%%%%%%%%%%%%%%%%%%%%%%%%%%%%%%%%%%%%%%%%%%%%%%%
%%%%%%%%%%%%%%%%%%%%%%%%%%%%%%%%%%%%%%%%%%%%%%%%%%%%%%%%%%%%%%%%%%%%%%%%%%%%%%%%%

\begin{thebibliography}{99}



\bibitem {AFT}  {\small \textsc{C. Amick, L. Fraenkel and J. Toland,} \ On the Stokes conjecture for the wave of extreme form,
{\it Acta Math.,} {\bf 148} (1982), 193-214.}

\bibitem{alla} {\small \textsc{B. Alvarez-Samaniego and D. Lannes,} \  Large time existence for 3D water-waves and asymptotics,
{\it Invent. Math.,} {\bf 171} (2008), 485-541.}

\bibitem{BMN1}{\small \textsc{A. Babin, A. Mahalov and B. Nicolaenko}, Regularity and integrability of 3D Euler and Navier-Stokes equations for rotating fluids, {\it Asymptot. Anal.}, {\bf 15} (1997) 103-150.}

\bibitem{BMN2}{\small \textsc{A. Babin, A. Mahalov and B. Nicolaenko}, Global regularity of 3D rotating Navier-Stokes equations for resonant domains, {\it Indiana Univ. Math. J.}, {\bf 48} (1999) 1133-1176.}

\bibitem{BMN3}{\small \textsc{A. Babin, A. Mahalov and B. Nicolaenko}, On the regularity of three-dimensional rotating Euler-Boussinesq equations, {\it Math. Models Methods Appl. Sci.}, {\bf 9} (1999) 1089-1121.}

\bibitem {BBM} {\small \textsc{T. Benjamin, J. Bona,  and J. Mahony,} \  Model equations for long waves in nonlinear dispersive media, {\it Phil. Trans. Roy. Soc. Lond. A,} {\bf 272} (1972), 47-78.}

\bibitem {BeOl} {\small \textsc{T. Benjamin and P. Olver,}  \ Hamiltonian structure, symmetries and conservation laws for water waves, {\it J. Fluid Mech.,} {\bf 125} (1982), 137-185.}


\bibitem{BCS} {\small \textsc{J.L. Bona, M. Chen, and J.C. Saut,} {\ Boussinesq equations and other systems for small amplitude long waves in nonlinear dispersive media. I. Derivation and linear theory,} {\it J. Nonlinear Sci.,} {\bf 12} (2002), 283-318}.


\bibitem{Bous} {\small
\textsc{J. Boussinesq,} {\ Th\'eorie g\'en\'erale des mouvements qui sont propag\'es dans un canal rectangulaire horizontal,} {\it Comptes Rendus Acad. Sci. Paris, } {\bf 73} (1871), 256-260.}

\bibitem{CH} {\small \textsc {R. Camassa and D. D. Holm,} \ An integrable shallow water equation with peaked solitons,
{\it Phys. Rev. Lett.,}  {\bf 71} (1993), 1661-1664.}


\bibitem{CDGG}{\small \textsc{J.-Y. Chemin, B. Desjardines, I. Gallagher and E. Grenier}, Mathematical geophysics, an introduction to rotating fluids and the Navier-Stokes equations, {\it Oxford Lecture Series in Mathematics and Its Applications}, vol. 32, Clarendon Press, Oxford (2006).}

\bibitem{Con12}{\small \textsc{A. Constantin,} \  On the modelling of equatorial waves, {\it Geophys. Res. Lett.}, {\bf 39} (2012) L05602.}

\bibitem{ce-1}{\small\textsc{A. Constantin and J. Escher,} \ Wave breaking for nonlinear
nonlocal shallow water equations, {\it Acta Math.}, {\bf 181} (1998), 229--243.}

\bibitem{ConJo}{\small \textsc{A. Constantin and R. Johnson}, The dynamics of waves interacting with the Equatorial Undercurrent, {\it Geophys. Astrophys. Fluid Dyn.}, {\bf 109} (2015) 311-358.}


\bibitem{CL09}{\small \textsc{A. Constantin and D. Lannes,} \ The hydrodynamical relevance of the Camassa-Holm and Degasperis-Procesi equations, {\it Arch. Ration. Mech. Anal.}, {\bf 192} (2009) 165-186.}

\bibitem{C85}{\small \textsc{W. Craig,} \ An existence theory for water waves and the Boussinesq and Korteweg-de Vries scaling limits, {\it Commun. Partial Differ. Equ.}, {\bf 10} (1985) 787-1003.}

\bibitem{CB}{\small \textsc{B. Cushman-Roisin and J. Beckers}, Introduction to geophysical fluid dynamics: physical and numerical aspects, {Academic Press} (2011).}


\bibitem{Danchin01}{\small \textsc{R. Danchin,} \ A few remarks on the Camassa-Holm equation,  {\it Differential Integral Equations}, {\bf 14} (2001) 953-988.}

\bibitem{DP}{\small\textsc{A. Degasperis and M. Procesi}, {\it Asymptotic integrability, in symmetry and perturbation theory}, edited by A. Degasperis and G. Gaeta, World Scientific, River Edge, New Jersey, (1999), 23-37.}

\bibitem{DuIs} {\small \textsc{V. Duch\^ene and S. Israw,} \ Well-posedness of the Green-Naghdi and Boussinesq-Peregrine systems, arXiv:1611.04305.}

\bibitem{EM}{\small \textsc{P. Embed and A. Majda}, Averaging over fast gravity waves for geophysical flows with arbitrary potential vorticity, {\it Commun. Partial Differ. Equ.}, {\bf 21} (1996) 619-658.}

\bibitem{Fan}{\small \textsc{L. L. Fan, H. J. Gao and Y. Liu}, On the rotation-two-component Camassa-Holm system modelling the equatorial water waves, {\it Adv. Math.}, {\bf 291}(2016),59-89.}

\bibitem{FB}{\small \textsc{A. Fedorov and J. Brown}, Equatorial waves. In {\it Encyclopedia of Ocean Sciences} (ed. J. Steele), Academic (2009) 3679-3695.}


\bibitem{GSR07}{\small \textsc{I. Gallagher and L. Saint-Raymond}, \ On the influence of the Earth's rotation on geophysical flows, {\it Handbook of Mathematical Fluid Mechanics}, {\bf 4} (2007) 201-329.}

\bibitem{Gard68}{\small \textsc{C. S. Gardner, M. D. Kruskal, and R. Miura}, \ Korteweg-de Vries equation and generalizations, II. Existence of conservation laws and constants of motion, {\it J. Math. Phys.}, {\bf 9} (1968) 1204-1209.}

\bibitem{GN76}{\small \textsc{A. Green and P. Naghdi}, \ A derivation of equations for wave propagation in water of variable depth, {\it J. Fluid Mech.}, {\bf 78} (1976) 237-246.}

\bibitem{GLL16}{\small \textsc{G. Gui, Y. Liu, and T. Luo,} A shallow-water modelling  with the Coriolis effect and travelling waves, submitted, 2017.}

\bibitem{GLS16}{\small \textsc{G. Gui, Y. Liu, and J. Sun,} A nonlocal shallow-water model arising from the full water waves with the Coriolis effect, submitted, 2017.}

\bibitem{Is2011}{\small \textsc{S. Israwi,} Large time existence for 1D Green-Naghdi equations, {\it Nonlinear Analysis: Theory, Methods \& Applications} {\bf 74} (2011) 81-93.}

\bibitem{Iv}{\small \textsc{R. Ivanov}, Two-component integrable systems modelling shallow water waves: the constant vorticity case, {\it Wave Motion} {\bf 46} (2009) 389-396.}

\bibitem{Iz}{\small \textsc{T. Izumo}, The equatorial current, meridional overturning circulation, and their roles in mass and heat exchanges during the El Ni\~no events in the tropical Pacific Ocean, {\it Ocean Dyn.} {\bf 55} (2005) 110-123.}

\bibitem{Jo1}{\small \textsc{R. S. Johnson,} \ Camassa-Holm, Korteweg-de Vries and related models for water waves, {\it Journal of Fluid Mechanics}, {\bf 455} (2002) 63-82.}

\bibitem{KN}{\small
\textsc{K. Kano and T. Nishida,} {\ A mathematical justification for Korteweg-de Vries equation and
Boussinesq equation of water surface waves,} {\it Osaka J. Math, } {\bf 23} (1986), 389-413.}

\bibitem{Ka-Pon-88}{\small \textsc{T. Kato and G. Ponce,}\  Commutator estimates and the Euler and Navier-Stokes equations, {\it Comm.
Pure App. Math.}, {\bf 41} (1988) 891-907.}


\bibitem{KdV} {\small \textsc{D. J. Korteweg and G. de Vries},  On the change of form of long waves advancing in a rectangular channel, and on a new type of long stationary waves,  {\it Phil. Mag.,} {\bf 39}, No. 5 (1895), 422-442.}

\bibitem {Li} {\small \textsc{Y. A. Li,} \ A shallow-water approximation to the full water wave problem, {\it Comm. Pure Appl. Math.,} {\bf 59} (2006), 1225-1285.}


\bibitem {Mc1} {\small \textsc{H. P. Mckean,} \ Breakdown of the Camassa-Holm equation, {\it
Comm. Pure Appl. Math.,} {\bf LVII} (2004), 0416-0418.}

\bibitem{Ped}{\small \textsc{J. Pedlosky}, Geophysical fluid dynamics, {\it Springer, New York}, 1992.}


\bibitem{Serre} {\small \textsc{F. Serre,} Contribution \`a l'\'etude des \'ecoulements permeanents et variables ands les canaux, {\it Houille Blanche,} {\bf 3} (1953) 374-388.}


\bibitem{SG} {\small \textsc{C. H. Su and C. S. Gardner,} Korteweg-de Vries equation and generalizations. III. Derivation of the Korteweg-de Vries equation and Burgers equation, {\it J. Math. Phys.,} {\bf 10} (1969) 536-539.}


\bibitem{Vallis}{\small \textsc{G. Vallis}, Atmospheric and oceanic fluid dynamics, {\it Cambridge University Press}, 2006.}

\bibitem {Wh} {\small \textsc{G. Whitham,} Linear and nonlinear waves, {\it John Wiley and Sons, New York,}  1973.}


\end{thebibliography}
\end{document}